\documentclass{article}

\usepackage{amsthm}
\usepackage{amssymb}
\usepackage{amsmath}
\usepackage{amscd}
\usepackage{makeidx}
\usepackage{fancyhdr}
\usepackage{systeme,mathtools}
\usepackage[all]{xy}
\usepackage{csquotes}
\usepackage[applemac]{inputenc}
\usepackage[T1]{fontenc}

\newcommand{\A}{\mathbb{A}}
\newcommand{\F}{\mathbb{F}}

\newcommand{\Q}{\mathbb{Q}}
\newcommand{\Z}{\mathbb{Z}}
\newcommand{\R}{\mathbb{R}}
\newcommand{\C}{\mathbb{C}}
\newcommand{\G}{\mathbb{G}}
\newcommand{\PP}{\mathbb{P}}

\newcommand{\calA}{{\cal A}}
\newcommand{\calB}{{\cal B}}
\newcommand{\calC}{{\cal C}}

\newcommand{\calE}{{\cal E}}
\newcommand{\calF}{{\cal F}}
\newcommand{\calH}{{\cal H}}
\newcommand{\calI}{{\cal I}}

\newcommand{\calL}{{\cal L}}
\newcommand{\calM}{{\cal M}}

\newcommand{\calO}{{\cal O}}
\newcommand{\calR}{{\cal R}}
\newcommand{\calS}{{\cal S}}
\newcommand{\calT}{{\cal T}}
\newcommand{\calU}{{\cal U}}
\newcommand{\calV}{{\cal V}}
\newcommand{\calW}{{\cal W}}
\newcommand{\calX}{{\cal X}}
\newcommand{\calY}{{\cal Y}}
\newcommand{\calZ}{{\cal Z}}

\def\AA{\mathbf{A}}
\def\BB{\mathbf{B}}
\def\CC{{\bf C}}
\def\EE{\mathbf{E}}

\newcommand{\kk}{\mathbf{k}}

\def\uw{{\underline{w}}}

\def\al{{\alpha}}
\def\be{{\beta}}
\def\de{{\delta}}

\def\De{{\Delta}}
\def\ep{{\epsilon}}
\newcommand{\vep}{{\varepsilon}}
\def\ga{{\gamma}}

\def\Ga{{\Gamma}}
\newcommand{\ka}{{\kappa}}
\def\la{{\lambda}}
\def\La{{\Lambda}}

\def\vpi{{\varpi}}
\def\si{{\sigma}}
\def\Si{{\Sigma}}
\def\th{{\theta}}
\def\vth{{\vartheta}}
\def\Th{{\Theta}}
\def\ups{{\upsilon}}
\def\ze{{\zeta}}
\def\om{{\omega}}
\def\Om{{\Omega}}
\newcommand{\oom}{\underline{\omega}}

\newcommand{\gotA}{\mathfrak A}

\newcommand{\gotB}{\mathfrak B}
\newcommand{\gotC}{\mathfrak C}

\newcommand{\gotH}{\mathfrak H}
\newcommand{\gotF}{\mathfrak F}
\newcommand{\gotm}{\mathfrak m}
\newcommand{\gotn}{\mathfrak n}
\newcommand{\gotoo}{\mathfrak o}
\newcommand{\gotp}{\mathfrak p}
\newcommand{\gotq}{\mathfrak q}
\newcommand{\gotR}{\mathfrak R}
\newcommand{\gots}{\mathfrak s}
\newcommand{\gotS}{\mathfrak S}
\newcommand{\gotg}{\mathfrak g}
\newcommand{\gotl}{\mathfrak l}

\def\can{{\rm can}}

\def\CM{{\rm CM}}

\def\End{{\rm End}}
\def\Gal{{\rm Gal}}
\def\GL{{\rm GL}}
\def\GO{{\rm GO}}
\def\GSp{{\rm GSp}}

\def\Hom{{\rm Hom}}
\newcommand{\Imm}{\mathrm{Im}}
\def\KS{{\rm KS}}

\def\M{{\rm M}}
\newcommand{\orn}{{\rm or}}
\newcommand{\per}{\text{per}}
\newcommand{\perr}{\text{per}_\infty}
\def\SL{{\rm SL}}
\def\SO{{\rm SO}}
\def\Sp{{\rm Sp}}
\newcommand{\Spec}{\mathrm{Spec}}
\newcommand{\Spf}{\mathrm{Spf}}
\def\Sym{\mathrm{Sym}}
\newcommand{\tra}{{\rm tr}}

\def\arH{H^{(\infty)}}

\def\Xar{{{\cal X}^1_\text{ar}}}
\def\XarD{{{\cal X}^\De_\text{ar}}}

\newcommand{\bs}{\backslash}

\newcommand{\derham}[1]{\calH^{#1}_{\mathrm{dR}}}

\newcommand{\dual}[1]{{#1}^\vee}

\newcommand{\mcd}[2]{\mathrm{gcd}(#1,#2)}

\newcommand{\ideles}{{K_{\A}^{\times}}}

\newcommand{\inv}[1]{{#1}^{-1}}

\newcommand{\invol}[1]{{#1}^\dagger}

\newcommand{\iisom}{\buildrel\sim\over\longrightarrow}
\newcommand{\isom}{\buildrel\sim\over\longrightarrow}

\newcommand{\map}{\rightarrow}
\newcommand{\mmap}{\longrightarrow}
\newcommand{\hmap}{\hookrightarrow}

\newcommand{\PFr}{{\rm Pr}_{\rm Frob}}

\def\smallmat#1#2#3#4{
    \left(\sopra{#1}{#3}\sopra{#2}{#4}\right)}
 
 \def\liminj#1{\mathrel{\mathop{\kern0pt\lim}
               \limits_{\kern3pt\longrightarrow\,#1}}}
                
 \def\limproj#1{\mathrel{\mathop{\kern0pt\lim}
                \limits_{\kern3pt\leftarrow\,#1}}}

\newcommand{\scal}[2]{\left\langle#1,#2\right\rangle}

\newcommand{\sopra}[2]{\genfrac{}{}{0pt}{1}{#1}{#2}}

\newcommand{\vass}[1]{|#1|}
\newcommand{\vvass}[1]{|\!\vass{#1}\!|}

\newcommand{\vvec}[2]{\left(\sopra{#1}{#2}\right)}

\newtheorem{dfn}{Definition}[section]
\newtheorem{thm}[dfn]{Theorem}
\newtheorem{rem}[dfn]{Remark}
\newtheorem{pro}[dfn]{Proposition}
\newtheorem{lem}[dfn]{Lemma}
\newtheorem{cor}[dfn]{Corollary}
\newtheorem{notat}[dfn]{Notation}

\begin{document} %%%%  BEGIN DOCUMENT

\markboth{A. Mori}{Power series expansions of modular forms}

\title{Power series expansions of modular forms\\ and $p$-adic interpolation of the square roots\\
of Rankin-Selberg special values}

\author{Andrea Mori\\ Dipartimento di Matematica\\ Universit\`a di Torino}
%\address{Dipartimento di Matematica\\
%          Universit\`a di Torino\\ Torino, Italia\\
%             \email{andrea.mori@unito.it}}
\date{}

\maketitle

\begin{abstract}
Let $f$ be a newform of even weight $2\ka$ for $D^\times$, where $D$ is a possibly split indefinite quaternion algebra over $\Q$.
Let $K$ be a quadratic imaginary field splitting $D$ and $p$ an odd prime split in $K$.
We extend our theory of $p$-adic measures attached to the power series expansions of $f$
around the Galois orbit of the CM point corresponding to an embedding $K\hookrightarrow D$ to forms
with any nebentypus and to $p$ dividing the level of $f$. 
For the latter we restrict our considerations to CM points  corresponding to test objects endowed with
an arithmetic $p$-level structure. Also, we restrict these $p$-adic measures to $\Z_p^\times$ and compute the corresponding
Euler factor in the formula for the $p$-adic interpolation of the \lq\lq square roots\rq\rq of the Rankin-Selberg
special values $L(\pi_K\otimes\xi_r,\frac12)$ where $\pi_K$ is the base change to $K$ of the automorphic representation
of $\GL_2$ associated, up to Jacquet-Langland correspondence, to $f$ and $\xi_r$ is a compatible family of
gr\"ossencharacters of $K$ with infinite type $\xi_{r,\infty}(z)=(z/\bar z)^{\ka+r}$.
\end{abstract}

\smallskip {Modular forms; automorphic $L$-functions.}


%%% INTRODUCTION

\section*{Introduction}

Let $f$ be a holomorphic modular newform of weight $2\ka$, level $N$ and nebentypus $\vep$  for $D^\times$, 
where $D$ is the  possibly split indefinite quaternion algebra of discriminant $\De$ with $(\De, N)=1$.
Let $K$ be a quadratic imaginary field splitting $D$ in which the odd prime $p$ splits, $p\calO_K=\gotp\overline\gotp$.
In  \cite{Mori11} we used the power series expansions of $f$ in terms of the Serre-Tate parameter  
around the Galois conjugates of the CM point $\tau\in\gotH$ associated to an embedding $\rho:K\hookrightarrow D$
to construct, under the hypotheses that $f$ has trivial nebentypus 
and $p$ is prime to the level $N$, a $p$-adic measure on $\Z_p$ whose $r$-th moment squared is strictly related
to the special Rankin-Selberg value $L(\pi_K\otimes\xi_r,\frac12)$.
Here $\pi_K$ is the base change to $K$ of the automorphic representation $\pi$ of $\GL_2$
defined by $f$ up to Jacquet-Langlands correspondence and $\xi_r$ is a certain family of 
gr\"ossencharacters of $K_\A^\times$.

The main goal of this paper is to compute the restriction of this measure to $\Z_p^\times$ or, in other words,
to identify the correct interpolating Euler factor. At the same time, we extend the theory to forms of any nebentypus 
and to primes dividing the level. While the former extension is formal, the latter requires  some adjustement.
Since modular and Shimura curves are not smooth mod $p$ for $p$ dividing the level we consider only 
$p$-ordinary split arithmetic test objects, essentially CM points
corresponding  to elliptic curves, or QM abelian surfaces, defined over the integers in a finite extension of $K_{\gotp}$
endowed with an arithmetic level structure at $p$ (see definition  \ref{th:testpair}).
%so to remain in the smooth locus of the curve, see \cite{KatMaz85}, \cite{Cla03}.
On these points the effects of the classical and $p$-adic operators $U$ and $V$
coincide. This allows to connect the effect of restricting the measure to $\Z_p^\times$ 
%on the power expansions of the form in terms of the Serre-Tate parameter (theorem \ref{th:mustar})
to the variation under the classical $V$ operator of the toric period
$J_r(f,\xi,\tau)=\int_{K_\A^\times/K^\times\R^\times}\phi_r(td_\infty^\tau)\xi(t)\,dt$,
where $\phi_r$ is the adelic lift of $\de_{2\ka}^{(r)}f$ ($\de_{2\ka}^{(r)}$ is the $r$-th iterate, in the automorphic sense, 
of the basic Maass operator), $d_\infty^\tau\in{\rm SL}_2(\R)$ is the standard parabolic matrix such that 
$d_\infty^\tau\cdot i=\tau$ and $\xi$ is a gr\"ossencharacter of $K_\A^\times$ of weight $(-\ka-r,\ka+r)$ that 
has a prescribed reduction over $\widehat\calO_{K,c}^\times$, where $c$ is the conductor relative to $N$ of the embedding $\rho$. 
The toric integrals $J_r(f,\xi,\tau)$ provide the link between the expansions
of $f$ and the special Rankin-Selberg values. On one side, they can be computed in terms of a coupling defined
on the $p$-adic space of $p$-adic forms for $\GL_1(K)$, which happens to be bounded on the closed space of avatars 
of forms of bounded level. On the other hand, their squares are linked to the special
Rankin-Selberg $L$-values, as originally proved by Waldspurger, \cite{Waldsp85}.
In doing so we follow \cite{HaKu91}, where each $J_r(f,\xi,\tau)^2$ is a special case of a generalized Fourier
coefficient for an automorphic form on $R_{K/\Q}\G_m^2$.
Building on results of Shimizu \cite{Shimi72} and refining the techniques of Waldspurger, Harris and 
Kudla use the seesaw identity associated with the theta 
correspondence between the similitude groups $\GL_{2}$ and $\GO(D)$ 
and the splitting $D=K\oplus K^{\perp}$ to express the generalized 
Fourier coefficients of a theta lift  $\theta_{\varphi}(F)$,
where $F\in\pi$ and $\varphi$ is a split primitive Schwartz-Bruhat 
function on $D_{\A}$, as Rankin-Selberg Euler products. 
The theory can be made more explicit 
using the local computations of Watson \cite{Wat03}
(see also \cite{Pra06}) that allow to pinpoint a precise Schwartz-Bruhat function outside a finite set of "bad" primes.
In section \ref{sec:explicit} we deal with some of the local cases that were not in the scope of
\cite{Wat03} and  \cite{Pra06}, in particular when the level is not square-free or when the local component 
of the central character is not trivial.

We now state the main result of the paper. 

\begin{thm}\label{th:main}
    Let $K\subset D$ be an embedding of conductor $c$ relative to $N$ with fixed point $\tau\in\gotH$
    corresponding to a $p$-ordinary split arithmetic test object, $y={\rm Im}(\tau)$. 
    Assume 
    % $(c,N)=1$ and 
    $(p,c)=1$ and $\calO_{K,c}^\times=\{\pm1\}$.
    For $r\geq0$ let $\chi_r$ be a gr\"ossencharacter  of $K_\A^\times$ of infinity type $\chi_{r,\infty}(z)=z^{2(\ka+r)}$
    %weight $(-2(\ka+r),0)$
    compatible with the central character $\vep$
    whose $p$-adic avatar is the $r$-th moment of a measure on $\Z_p$ with values in the $p$-adic space
    $\gotF(c^*,F)$ for a sufficiently large local extension $K_\gotp\subset F$. 
    Let $\xi_r=\chi_r\vvass{\text{N}_{K/\Q}}^{-\ka-r}$.
    
    Then there exists a $p$-adic measure $\mu$ on $\CC_c^\sharp\times\Z_p^\times$ such that
   \begin{multline*}
    \frac1{\Om_p^{4(\ka+r)}}\left(\int_{\CC_c^\sharp\times\Z_p^\times}\psi(s)t^r\,d\mu(s,t)\right)^2=
    \frac{\ups}{({\mathrm{vol}}(\hat\calO_{K,c}^\times)h_c^\sharp)^2}\frac{\pi^{2\ka+r-3}}{(y\Om_\infty^2)^{2(\ka+r)}}\times\\
    E_\gotp(\psi\xi_r)^2\La(\varphi^\sharp_r,\psi\xi_r)
    L(\pi_K\otimes\psi\xi_r,\frac12)L(\eta_K,1)^{-1}
    \end{multline*}
    where
    \begin{enumerate}
       \item $\ups$ is a normalizing ratio of Petersson norms, $\Om_\infty$ and $\Om_p$ are archimedean and 
                 $p$-adic periods respectively;
       \item $E_\gotp(\psi\xi_r)=(1-a_p(\psi\xi_r)_\gotp(\vpi_\gotp)p^{-2\ka}+\vep(p)(\psi\xi_r)_\gotp(\vpi_\gotp)^2p^{-2\ka-1})$,
                 where $a_p$ is the $p$-th Hecke eigenvalue of $f$ and $\vpi_\gotp$ is a local uniformizer in $K_\gotp$
                 well defined up to a unit in the conductor of $\xi_{r,\gotp}$;
       \item the factor $\La(\varphi_r,\psi\xi_r)\in\overline\Q$ is a product of local terms almost all of them equal to $1$:
                it depends on an explicit Schwartz-Bruhat function 
                $\varphi^\sharp_r=\varphi^{\sharp(r,r)}_\infty\otimes\bigotimes_{\ell<\infty}\phi^\prime_\ell$
                which varies with $r$ only in the archimedean component.
    \end{enumerate}
\end{thm}

The compatibility of $\chi_r$, hence $\xi_r$, with the central character, which had been already hinted in the introductory remarks, 
is in the sense of proposition \ref{pr:Jasscpr}: in parti\-cu\-lar it is required
that ${\xi_r}_{|\A^\times}=\vep^{-1}$. One should notice that the compatibility condition
depends strongly on $\rho$, so, in general, changing the embedding results in a $p$-adic interpolation of an entirely
different set of twists.
A standard way to produce a family of gr\"ossencharacters $\xi_r$ as requested is outlined in remark \ref{rem:onchar}.

The link with the couplings discussed in section \ref{se:moremeas}, which has already been alluded to above,
is also a strong clue for the non-vanishing of possibly  many of the special values
$L(\pi_K\otimes\psi\xi_r,\frac12)$ as $\xi_r$ is kept fixed and $\psi$ varies in $\widehat{\CC}_c^\sharp$.
In order to be able to give a quantitative version of this expectation, one needs a tight control of the values
of $\de^r(f)$ at the Galois orbit of the relevant CM point.

The idea that the coefficients of the power series expansions at CM points of a modular form  
(that had been introduced in \cite{Mori94} with a different goal)
could be used to interpolate $p$-adically special values $L(\pi_K\otimes\xi,\frac12)$
arose a long time ago in private conversations with M. Harris and was mentioned for the first time in $\cite{HaTi01}$.
Not long after \cite{Mori11} came out in print and already a good deal of the material for the present paper had been
worked out we became aware of the works of M. Bertolini, H. Darmon and K, Prasanna \cite{BeDaPr13}
and M. Brako\v{c}evi\'c \cite{Bra11} where such $p$-adic interpolations had been obtained. 
The main goal  of \cite{BeDaPr13} is far beyond the mere construction of a $p$-adic $L$-function and the authors
work under less inclusive hypotheses and the techniques used appear substantially different.
On the other hand the main goal of  \cite{Bra11} is to construct a "cuspidal" measure in the spirit
of Katz's classical paper \cite{Katz78}, a goal that turned out to be tightly related to our looking for arithmetic
applications of the power series expansions other than a mere analogue of the classical $q$-expansion principle 
(valid only in the split case).
Although the techniques employed are not exactly  the same and the main result is somewhat different, 
there is a good overlap between the ideas employed in  Brako\v{c}evi\'c's paper  and in here.
While Brako\v{c}evi\'c's main result is more invariant and better fitted for applications
as he constructs a unique measure on the profinite group ${\rm Cl}_K^-(N_{\rm ns}p^\infty)$ (his notation, where
$N_{ns}$ is the non-split part of the level $N$) that interpolates a large deal of twists, our approach makes
an explicit use of the power series expansions at points in the non-split Shimura curves. Since the use of
expansions in the anisotropic case seems interesting per se and to the best of our knowledge unprecedented,
we decided, even after such long time, to complete the present work in its final form.

%
%Finally, the work \cite{DaGu08} of Datskovsky and Guerzhoy should be mentioned.
%Working in much less generality (they treat only the split case under some 
%technical conditions on $p$) and from the complex analytic point of view, they manage 
%to obtain Kubota-Leopoldt kind of congruences for the expansion coefficients. 

\subsection*{Notations and Conventions.}

Prime numbers are denoted $\ell$ 
(possibly $\ell=\infty$ with the convention that $\Q_\infty=\R$, depending on context). 
When a finite prime number is to be supposed fixed we denote it $p$ and we
fix once for all embeddings $i_\infty\colon\overline{\Q}\map\C$ and 
$i_p\colon\overline{\Q}\map\C_p$.
If $v$ is a  finite place in the number field $L$ with ring of integers  $\calO_{L}$, 
we denote $\calO_{(v)}$ the ring of $v$-integers in $L$, while
$L_{v}$, $\calO_v$ and $k_v$ are the $v$-adic completion of $L$, of $\calO_L$ 
and the residue field respectively.
A quadratic imaginary field is always denoted $K$ and if $c$ a positive integer, 
$\calO_{K,c}=\Z+c\calO_{K}$ is its order of conductor $c$. If  $\gotn\subset\calO_K$
is an ideal denote $\CC_\gotn=\calI_\gotn/P_\gotn=K_\A^\times/K^\times\C^\times U_\gotn$ the group of $\gotn$-ideal classes of $K$
and $h_\gotn=|\CC_\gotn|$.
Also, let $\CC_c^\sharp=\CC_{c\calO_K}^\sharp=K_\A^\times/K^\times\C^\times\widehat{\calO}_{K,c}^\times$
and $h_c^\sharp=|\CC_c^\sharp|$.
%Also, we denote $\gotn^\ast$ the prime-to-$p$ part of $\gotn$

Fix an additive character $\psi$ of $\A/\Q$ ($\A=\Q\R\widehat\Z$ is the ring of rational adeles)
asking that $\psi_\infty(x)=e^{2\pi i x}$
and $\psi_\ell$ is trivial on $\Z_\ell$ with $\psi_\ell(x)=e^{2\pi i x}$ for $x\in\Z[\inv\ell]$ and finite $\ell$. 
This determines the Haar measure $dx=\prod_{\ell\leq\infty}dx_\ell$ on $\A$, where the local Haar 
measures $dx_\ell$ are normalized so that the $\psi_\ell$-Fourier transform is autodual. 
On $\A^\times$ we fix the multiplicative Haar measure $d^\times x=\prod_{\ell\leq\infty}d^\times x_\ell$ 
where $d^\times x_\ell=\ze_\ell(1){\vass x_\ell}^{-1}dx_\ell$, so that ${\rm vol}(\Z_\ell^\times)=1$ for finite $\ell$.
For a quaternion algebra $D$  the Haar measure 
$dx=\prod_{\ell\leq\infty}dx_\ell$ on $D_\A$ is fixed requiring that
the local measures $dx_\ell$ are normalized so 
that the Fourier transform with respect to the norm form is autodual. 
On $D(\A)^\times$ we fix the multiplicative measure $d^\times x=\prod_\ell d^\times x_\ell$ where
$d^\times x_\ell=\ze_\ell(1){\vass x_\ell}^{-1}dx_\ell$ on $(D\otimes\Q_\ell)^\times$,
so that ${\rm vol}(\GL_2(\Q_\ell))=\ze_\ell(2)^{-1}$ when $D=\GL_2$.

Given a quadratic space $(V,\scal{\,}{\,})$ of dimension $d$ over $\Q$
denote $\calS_{\A}(V)=\bigotimes_{\ell\leq\infty}\calS_{\ell}$ the adelic 
Schwartz-Bruhat space: for $\ell$ finite, $\calS_{\ell}$ is the space of 
locally constant functions on $V\otimes\Q_\ell$ with compact support 
and $\calS_{\infty}$ is the space of Schwartz functions on 
$V\otimes\R$ which are finite under the natural action of a (fixed) maximal compact subgroup of 
the similitude group $\GO(V)$. The Weil representation $r_\psi$ is the representation of 
$\SL_2(\A)$ on $\calS_{\A}(V)$ which is explicitly described locally at $\ell\leq\infty$ by

\begin{subequations}\label{eq:Weilrep}
\begin{eqnarray*}
r_\psi\left(\begin{array}{cc}1 & b \\0 & 1\end{array}\right)\varphi(x) & = & 
\psi_\ell\left(\frac12\scal{bx}{x}\right)\varphi(x), \\
r_\psi\left(\begin{array}{cc}a & 0 \\ 0 & \inv a\end{array}\right)\varphi(x) & = & 
\chi_V(a)\vass{a}_\ell^{d/2}\varphi(ax) \\
r_\psi\left(\begin{array}{cc}0 & 1 \\ -1 & 0\end{array}\right)\varphi(x) & = & \gamma_V\hat{\varphi}(x)
\end{eqnarray*}
\end{subequations}
where $\ga_V$ is an eighth root of 1 and $\chi_V$ is a quadratic character that are computed 
in our cases of interest in \cite{JaLa70} (see also \cite[\S3.4]{Pra06}), while 
$\hat{\varphi}(x)=\int_{V\otimes\Q_\ell}\varphi(y)\psi_\ell(\scal xy)\,dy$ is the Fourier transform 
computed with respect to a $\scal{\,}{\,}$-self dual Haar measure on $V\otimes\Q_\ell$.

The group $\GL_2^+(\R)$ acts on the upper half-plane $\gotH=\{\tau=x+yi\in\C\,|\,y>0\}$
by linear fractional transformations: $\smallmat abcd\cdot \tau=\frac{a\tau+b}{c\tau+d}$.
The automorphy factor is $j(g,\tau)=c\tau+d$.
Given $\tau=x+yi\in\gotH$ let $d^\tau_\infty=\sqrt{y}\smallmat{1}{x/y}0{1/y}\in\SL_2(\R)$
the upper triangular matrix such that $d^\tau_\infty\cdot i=\tau$.
The space of holomorphic modular forms with respect to the Fuchsian groups of the first kind
$\Ga_\ast^\De(N)$, $\ast\in\{0,1\}$, whose definition is recalled
in section \ref{backMSC}, will be denoted $M_{\ast,k}^\De(N)$. The bigger spaces of 
$\calC^\infty$-modular forms for whom only analiticity is required, mantaining the growth conditions
at the cusps, will be denoted $M_{\ast,k,\infty}^\De(N)$.

%%%%%%%%%%%%%%%%%%% CHAPTER ONE

\section{Background}

\subsection{Modular and Shimura curves.}\label{backMSC}
Let $D$ be an indefinite quaternion algebra over $\Q$ with reduced norm $\nu$ 
and reduced trace $\tra$. 
For every rational place $\ell$ let $D_\ell=D\otimes\Q_\ell$. Fix an isomorphism
$\Phi_\infty\colon D_\infty\isom\M_2(\R)$ which will be usually left implicit.
Let $\Si_D$ be the finite and even set of places at which $D$ is ramified and let 
$\De=\De_D=\prod_{\ell\in\Si_D}\ell$ be the discriminant of $D$. 
If $\calR\subset D$ is an order and $\ell<\infty$  let $\calR_\ell=\calR\otimes\Z_\ell$.
Fix:
\begin{itemize}
  \item A maximal order $\calR\subset D$, with isomorphisms
           $\Phi_\ell\colon D_\ell\isom\M_2(\Q_\ell)$ for all finite $\ell\notin\Si_D$ such that
           $\Phi_\ell(\calR_{\ell})=\M_2(\Z_\ell)$.
           For every $N$ such that $(N,\De)=1$ let $\calR_N\subset\calR$ 
           be the Eichler order given by 
           the local conditions
           $$
          \Phi_\ell(\calR_{N,\ell})=
          \left(\left\{\left(
          \begin{array}{cc}
	a & b  \\
	c & d
	 \end{array}\right)
          \hbox{$\in\M_2(\Z_\ell)$ such that $c\equiv0\bmod N$}
          \right\}\right)
          $$
          for $\ell\notin\Si_D$, and $\calR_{N,\ell}$ is the unique
          maximal order in $D_\ell$ for $\ell\in\Si_D$.
          It is equipped with homomorphisms ($\ell$-orientations)
          $\orn_\ell^{1},\orn_\ell^{2}\colon\calR_{N}\otimes\F_\ell\map\F_{\ell^2}$ 
          for $\ell\mid\De$, and 
          $\orn_\ell^{1},\orn_\ell^{2}\colon\calR_{N}\otimes\Z/\ell^s\Z\map({\Z/\ell^s\Z})^2$ 
          for  $\ell^s\mid\mid N$. 
          If $D=\M_2(\Q)$ we take $\calR=\M_2(\Z)$.
  \item A positive involution $d\mapsto\invol d=\inv t\bar{d}t$
         with $t\in\calR$ and $t^2=-\De$. Thus $\invol\calR=\calR$ and the skew-symmetric
         bilinear form $B_t(a,b)=\tra(a\bar{b}t)=\tra(at\invol b)$ is
         non-degenerate and $\Z$-valued on $\calR\times\calR$. 
\end{itemize}

\medskip\noindent
Consider the groups
\begin{gather*}
\Ga_0^\De(N)=\calR^{1}_{N}=\left\{
\hbox{$\ga\in\calR_{N}$ such that $\nu(\ga)=1$}
\right\}\quad\text{and}\\
\Ga_1^\De(N)=\left\{\hbox{$\ga\in\calR^{1}_{N}$ such that 
$\orn^{\ep}_{\ell}(\bar\ga)=1$
 for $\ell\mid N, \ep=1,2$}\right\}.
 \end{gather*}
When $\De=1$, $\Ga_0^1(N)$ and $\Ga_1^1(N)$ are the classical 
congruence subgroups $\Ga_0(N)$ and $\Ga_1(N)$
%\begin{gather*}
%\Ga_0(N)=\left\{
%\left(
%\begin{array}{cc}
%	a & b  \\
%	c & d
%\end{array}
%\right)
%\in\SL_2(\Z)\hbox{ such that $c\equiv0\bmod N$}\right\}\quad\text{and}\\
%\Ga_1(N)=
%\left\{
%\left(
%\begin{array}{cc}
%	a & b  \\
%	c & d
%\end{array}
%\right)
%\in\SL_2(\Z)\hbox{ such that $a,d\equiv1$ and $c\equiv0\bmod N$}\right\}
%\end{gather*}
respectively. The groups $\Ga_\ast^\De(N)$ are, 
via $\Phi_\infty$, discrete subgroups of 
$\SL_2(\R)$ acting on $\gotH$.
When $\De>1$ the quotient $X_{\ast}^\De(N)=\Ga_\ast^\De(N)\bs\gotH$ is a 
compact Riemann surface. When $\De=1$
let $X_\ast^1(N)$ be the standard cuspidal 
compactification of $Y_\ast(N)=\Ga_\ast(N)\bs\gotH$.
Each of these curves $X^\De_\ast(N)$  is the set of 
complex points of a proper scheme $\calX^\De_\ast(N)$ smooth over $\Z[1/{N\De}]$ 
which is the solution, for $N$ not too small and $\ast=1$, of a representable moduli 
problem defined over $\Z[1/\De]$.

When $\De=1$, $N>3$ the scheme $\calX_1^1(N)$ is the compactified moduli space 
for the functor $F_1^1(N)\colon\hbox{\bf Schemes}\map\hbox{\bf Sets}$
defined by 
\begin{equation}%\label{eq:functor}
F_1^1(N)(S)=
\left\{
\parbox{2.3 in}{
Isomorphism classes of elliptic curves $E_{/S}$ with a $\Ga_1(N)$-structure.
}\right\}
\end{equation}
A $\Ga_1(N)$-structure on an elliptic curve $E_{/S}$ is a section $P\colon S\map E$ 
such that the effective Cartier divisor $\sum_{d\in\Z/N\Z}dP$ is a subgroup scheme of $E$.
Denote $\pi_N\colon\calE_{N}\map\calX_1^1(N)$
the corresponding universal (generalized) elliptic curve.  
For $S=\Spec(\C)$ the fiber corresponding to $z\in\gotH$ is the torus 
$E_z=\C/\Z\oplus\Z z$ with point $P=1/N \bmod\Z$. 
When $\De>1$ and $N>3$, the scheme $\calX_1^\De(N)$ 
represents the functor
$F_1^\De(N)\colon\hbox{\bf $\Z[1/\De]$-Schemes}\map\hbox{\bf Sets}$ 
defined by 
\begin{equation}%\label{eq:functorD}
F_1^\De(N)(S)=
\left\{
\parbox{3 in}{Isomorphism classes of compatibly principally
                        polarized abelian surfaces $A_{/S}$ with a 
                        ring embedding $\iota:\calR\hmap\End_S(A)$ and
                        a $\Ga_1(N)$-structure} 
\right\}.
\end{equation}
An abelian surface $A_{/S}$ with a ring embedding $\iota:\calR\hmap\End_S(A)$ 
is said to have quaternionic multiplications and is called a 
QM-\emph{abelian surface} (or \emph{false elliptic curve}).
A principal polarization on a QM-abelian surface is 
compatible with the embedding $\iota:\calR\subset\End_S(A)$ if the 
involution $d\mapsto\invol d$ is the Rosati involution.
Given our choice of involution, there exists a unique compatible principal 
polarization on a given QM-abelian surface $(A,\iota)$.
A $\Ga_1(N)$-structure on a QM-abelian surface $A_{/S}$ is the datum of
a section $P:S\map A$ such that the effective Cartier divisor 
$\sum_{d\in\Z/N\Z}dP$ is a rank $N$ subgroup of
$\ker\left(\smallmat1000:A[N]\map A[N]\right)$ where 
$\smallmat1000\in\M_2(\Z/N\Z)\simeq\calR\otimes(\Z/N\Z)$ acts on $A[N]$ via endomorphisms.
Denote $\pi_{\De,N}\colon\calA_{\De,N}\mmap\calX_1^\De(N)$
the corresponding universal QM-abelian surface.
For $S=\Spec(\C)$ the fiber corresponding to $z\in\gotH$ is the torus 
$A_z=D_\infty^z/\calR$ where $D_\infty^z$ denotes the real vector space $D_\infty$ endowed with the 
$\C$-structure defined by the identification 
$\C^2=\Phi_\infty(D_\infty)\left(\sopra{z}{1}\right)$, i.e. 
$A_z=\C^2/\Phi_\infty(\calR)\vvec{z}{1}$. 
%The complex 
%uniformization defines a na\"ive full $N$-level structure 
%$\frac1N\calR/\calR\isom A_{z}[N]$. 
The skew-symmetric form
$\scal{\Phi_\infty(a)\left(\sopra{z}{1}\right)}
{\Phi_\infty(b)\left(\sopra{z}{1}\right)}=B_t(a,b)$ 
extended to $\C^{2}$ by 
$\R$-linearity is the unique Riemann form on $A_{z}$ with Rosati 
involution $d\mapsto\invol d$. 

The schemes $\calX_0^\De(N)$ are
constructed from the schemes
$\calX_1^\De(N)$  by taking quotients by the
action of the quotient group $\Ga_1^\De/\Ga_0^\De\simeq\left(\Z/N\Z\right)^\times$
which can be described modularly as $P\mapsto dP$ (diamond operators).
They are the coarse moduli schemes attached to the functor
$F_0^\De(N)$ obtained replacing the  section $P$ in the moduli problem datum
by a closed locally free cyclic subgroup scheme of rank $N$ which
in the case $\De>1$ needs to be a subgroup scheme of 
$\ker\left(\smallmat1000:A[N]\map A[N]\right)$.

When $\ell\mid N$ the structure of the non-smooth fibers $\calX_\ast^\De(N)\otimes\F_\ell$ 
are very similar and have been determined in 
 \cite{KatMaz85} when $\De=1$ and in \cite{Buz97, Cla03} adapting the techniques of 
 Katz and Mazur to the case of QM-abelian surfaces. 
 Write $N=N^\prime\ell^s$ with $(N^\prime,\ell)=1$.
Then
$$
\calX_1^\De(N^\prime\ell^s)\otimes\F_\ell=\coprod_{r=0}^s\calZ_1^\De(r)
$$
where the component $\calZ_1^\De(r)$ has multiplicity 
$\phi(\ell^{s-r})$. The reduced scheme $\calZ_1^\De(r)^\text{red}$ is the
Igusa curve $\calX_1^\De(N^\prime,{\rm Ig}[\ell^r])_{/\F_\ell}$
(cuspidally compactified when $\De=1$), i.e.
the smooth $\F_\ell$-scheme that represents the functor
${\rm Ig}(F_1^\De(N^\prime),\ell^r)$ on $\F_\ell$-schemes defined, when $\De=1$, by
$$
{\rm Ig}(F_1(N^\prime),\ell^r)(S)=
\left\{
\parbox{2.8 in}{Isomorphism classes of elliptic curves $E_{/S}$ with a 
$\Ga_1(N^\prime)$-structure and a choice of a generator
of $\ker({\rm Ver}^r\colon E^{(\ell^r)}\map E)$.} 
\right\}.
$$
and similarly in the $\De>1$ case, where
an Igusa $\ell^r$-structure on a QM-abelian surface $A_{/S}$ 
is a choice of generator of $\smallmat1000\ker({\rm Ver}^r\colon A^{(\ell^r)}\map A)$. 
Moreover, in either case all the components meet transversally at each 
supersingular point.
The description of $\calX_0^\De(N^\prime\ell^s)\otimes\F_\ell$ is completely analogous, 
\cite[\S 13.5.6]{KatMaz85}.

\begin{rem}
\rm  Let $\XarD(N)\hmap\calX_1^\De(N)$ be the smooth open subscheme
obtained by discarding the non-reduced components at the primes $\ell\mid N$.
The above modular description in terms of Igusa curves entails, by Cartier duality, 
that $\Xar(N)$ represents the functor on schemes
$$
F^1_\text{ar}(N)(S)=
\left\{
\parbox{2.8 in}{Isomorphism classes of elliptic curves $E_{/S}$ with an 
                        embedding $\mu_N\hmap E$ of $S$-group schemes} 
\right\}
$$
and, when $\De>1$, $\XarD(N)$ represents the functor on $\Z[1/\De]$-schemes
$$
F_\text{ar}^\De(N)(S)=
\left\{
\parbox{2.5 in}{Isomorphism classes of QM-abelian surfaces $A_{/S}$ with an 
                           $\calR$-equivariant embedding 
                           $\mu_N\times\mu_N\hmap A$ of $S$-group schemes} 
\right\}
$$
These structures are called $N$-level arithmetic structures.
\end{rem}

\medskip\noindent
\subsection{CM points.}\label{se:cmpts}
Let $\Q\subseteq E\subset F$ be fields with $[F:E]=2$, 
$F=E(\al)$ with $\al^2=a\in E$, and $F$ splitting $D$.
An embedding $\rho:F\hmap D\otimes_\Q E$ endows $D\otimes_\Q E$ 
with a $F$-vector space structure: scalar multiplication by 
$\la\in F$ is left multiplication by $\rho(\la)$. 
If $\Gal(F/E)=\{1,\si\}$ let $\rho^\si=\rho\circ\si$. By the Skolem-Noether theorem 
there exists $u\in(D\otimes_\Q E)^\times$, well defined up to a 
$F^\times$-multiple, such that $u\rho(\la)=\rho^\si(\la)u$ 
for all $\la\in F$, and $u^2\in E$. With a slight abuse of notation
a splitting $D\otimes_\Q E=F\oplus Fu$ is obtained,
more intrinsically seen either as the eigenspace decomposition 
under right multiplication by $\rho(F^\times)$ or an orthogonal decomposition
under the norm form.
The projection onto $F$ with kernel $F^\perp=Fu$ is  the idempotent
$e_\rho$ image of 
$\frac{1}{2}\left(1\otimes1+\frac1a\al\otimes\al\right)$ under 
$$
F\otimes_EF\stackrel{\rho\otimes\rho}\mmap
(D\otimes_\Q E)\otimes_E (D^{\rm op}\otimes_\Q E)\simeq{\rm End}_E(D\otimes_\Q E).
$$
An involution $d\mapsto d^\imath$ in $D$ extends by linearity to $D\otimes_\Q E$. 
If $\rho^\imath$ is the embedding $\rho^\imath(\la)=\rho(\la)^\imath$, then
$e_{\rho^\imath}=(e_\rho)^\imath$ and in particular
$(e_\rho)^\imath=e_\rho$ if and only if $\rho(F)^\imath=\rho(F)$ pointwise.

\bigskip
Let $E=\Q$ from now on. The positivity of the involution $d\mapsto\invol d$ implies the existence of a
traceless element $\de\in D$ such that $\de=\invol\de$, so that $F=\Q(\de)\subset D$
is a real quadratic subfield pointwise fixed by the involution and the corresponding 
idempotent $e\in D\otimes_\Q F$ satisfies $\invol e=e$.

\bigskip
Now let $F=K$, a quadratic imaginary field. In 
each $\Gal(K/\Q)$-orbit in $\Hom(K,D)$ there is exactly one embedding
which is normalized in the sense of \cite[(4.4.5)]{ShiRed}.
The normalized embeddings correspond bijectively to a special set 
of points $\tau\in\gotH$. More precisely, there is a bijection
between the following two sets:
\begin{enumerate}
  \item $\Hom^\sharp(K,D)=\left\{\hbox{normalized embeddings $\rho\colon K\hmap D$}\right\}$.
  \item $\CM_{\De,K}=\left\{\hbox{$\tau\in\gotH$ such that $\mathrm{Stab}(\tau)\cap\Phi_\infty(D^\times)
            \simeq K^\times$}\right\}$.
\end{enumerate}
The bijection is $\Ga_0^\De(N)$-equivariant where $\Ga_0^\De(N)$ acts by 
conjugation on the first set and on $\CM_{\De,K}$ via its action on $\gotH$.
Under the correspondence $\rho\leftrightarrow\tau$ 
the complex structure on $D_\infty$ induced by the 
embedding $\rho$ coincides with that of $D_{\infty}^{\tau}$.
In the split case $\CM_{1,K}=K\cap\gotH$. 

\bigskip
The conductor $c_\rho=c_\rho(N)$ relative to the order $\calR_N$
of the embedding $\rho\in\Hom(K,D)$
is the unique $c\in\Z^{>0}$ such that   
$\rho(\calO_{K,c})=\rho(K)\cap\calR_{N}$.
The conductor $c=c_\rho$ of the embedding associated to the point $\tau\in\CM_{\De,K}$
depends only on the point 
$x_\tau=[\tau]\in\Ga_0^\De(N)\bs\gotH\subset X_0^\De(N)$,
which will be called a CM point of type $K$ and conductor $c$. We let
$\CM(\De,N;K,c)$ denote the set of CM points of type $K$ and 
conductor $c$.
When $\mcd c{N\De}=1$ the set $\CM(\De,N;K,c)$ is non-empty if and only if
all primes $\ell\mid\De$ are inert in $K$ and all primes $\ell\mid N$ are split in $K$.

The abelian variety $A_x$ corresponding to $x\in\CM_{\De,K}$ has a large
set of endomorphisms. When $\De=1$, $A$ is a elliptic curve with 
complex multiplications in $K$.
When $\De>1$ the QM-abelian surface $A$  contains the CM-elliptic curve
$E=K\otimes\R/\calO_{K,c_1}$ 
and is in fact isogenous to the product $E\times E$.

\begin{dfn}\label{th:testpair}
   Suppose  $\De>1$. Let $p$ be a prime number, $\mcd p{2\De}=1$. 
   A \emph{$p$-ordinary triple} for $(\De,N,K)$ is a 
   triple $(x, v, e)$, where $x\in\CM(\De,N;K,c)$, 
   $v$ is a finite place dividing $p$ in a finite extension $L\supseteq\Q$
   and $e\in D\otimes F$ is the idempotent associated to a real 
   quadratic subfield $F\subset D$ pointwise fixed by the positive 
   involution, such that
     \begin{enumerate}
       \item $FK\subseteq L$ and $x_\tau\in\calX_1^\De(N)(L)$; 
       \item the QM-abelian surface $A_x$ has ordinary good reduction modulo $v$;
       \item if $w$ is the restriction of $v$ to  $F$ then $e\in\calR\otimes_{\Z}\calO_{(w)}$. 
    \end{enumerate}
    Furthermore, a $p$-ordinary test triple $(x, v, e)$ is  said to be
    \emph{split} if $p$ splits in $F$, and is said to be \emph{arithmetic} if
    $x\in\XarD(N)(\calO_v)$.
\end{dfn}
   
\begin{rem}
   \rm \begin{enumerate}
       \item The ordinarity hypothesis implies that $p$ splits in $K$. We shall denote $\gotp\subset\calO_K$ the
                prime ideal corresponding to the place of $K$ defined by the restriction of $v$, so that $p\calO_K=\gotp\bar\gotp$
                where $\bar\gotp$ is the other ideal of $\calO_K$ of norm $p$.
       \item The third condition is equivalent to $(p,c_F\de_F)=1$, where $c_F$
                 is the conductor of $F\subset D$ relative to $\calR$.
       \item In the split case ($\De=1$) the field $F$ and the idempotent $e$ play no role
                 and only \emph{pairs} $(x,v)$ will be relevant, where $x\in\CM(1,N;\calO_{K,c})$
                 and $v\mid p$ is some place of ordinary good reduction for the elliptic curve $E_x$
                 in the number field $L$. Speaking indifferently of either case,  the term
                 \emph{$p$-ordinary (split, arithmetic) object} shall be used.  
      \end{enumerate}
\end{rem}

\begin{pro}
Let $p$ be a prime number such that $(p,2\De)=1$. 
There exists a real quadratic field $F$ in which $p$ splits and an embedding $F\hookrightarrow D$
of conductor prime to $p$ whose image is pointwise fixed by the involution $\invol{(\cdot)}$.
\end{pro}

\begin{proof}
Since maximal orders in $D$ are conjugated, it is enough to work with the Hashi\-mo\-to model 
\cite{Hash95} of $(D,\calR,t)$, namely
$$
\left\{
   \begin{aligned}
     D &  =\Q\oplus\Q i\oplus\Q j\oplus\Q ij,\qquad i^2=-\De,\quad j^2=q \\
     \calR &  =\Z\oplus\Z\frac{1+j}2\oplus\Z\frac{i+ij}2\oplus\Z\frac{b\De j+ij}q
   \end{aligned}
\right.
$$
and $t=i$, where $q$ is an auxiliary prime such that $q\equiv1\bmod4$ ($q\equiv5\bmod8$ when 
$2\mid\De$) and $(q,-\De)_\ell=-1$ if and only if $\ell\mid\De$ 
(such prime exists for Dirichlet's theorem of primes in arithmetic sequence),
and $b^2\De\equiv-1\bmod q$.
Imposing further that $\bigl(\frac qp\bigr)=1$, the request is met by  $F=\Q(\sqrt{q})$
embedded as $\sqrt{q}\hookrightarrow j\in\calR$, for which $c_F=1$.
\end{proof}

\subsection{Differential operators.}\label{diffop}
Let $T$ be a scheme, $S$ a smooth $T$-scheme 
and $\pi:\calA\map S$ an abelian scheme 
with $0$-section $e_{0}$ and dual $\calA^t_{/S}$. Let
$\oom=\oom_{\calA/S}=\pi_*\Om^1_{\calA/S}={e_0}^*\Om^1_{\calA/S}$
be the sheaf of translation invariant relative $1$-forms on $\calA$
and $\derham{1}=\derham{1}(\calA_{/S})=\R^1\pi_*(\Om^\bullet_{\calA/S})$
the (first) de Rham sheaf. They are two of the terms in the exact sequence
\begin{equation}
0\mmap\oom\mmap\derham{1}\mmap R^1\pi_*\calO_{\calA}\mmap0.
\label{eq:Hodgeseq}
\end{equation}
of sheaves on $S$ (the Hodge sequence). 
The Kodaira-Spencer map
$$
\KS:\oom\mmap R^1\pi_*(\pi^*\Om^1_{S/T})\simeq
\Om^1_{S/T}\otimes R^1\pi_*\calO_{\calA}.
$$
is the first boundary map in the long exact sequence arising by pushing down
the canonical exact sequence
$0\map\pi^*\Om^1_{S/T}\map\Om^1_{\calA/T}\map\Om^1_{\calA/S}\map 0$. 
By Serre duality, it can be seen as an element  of
$\Hom_{\calO_S}(\oom_{\calA/S}\otimes\oom_{\calA^t/S},\Om^1_{S/T})$
and can be reconstructed from the Gau{\ss}-Manin connection
$\nabla\colon\derham{1}\map\derham{1}\otimes_{\calO_S}\Om^1_{S/T}$
(which is the differential $d_1^{0,q}$ 
in the spectral sequence attached to the filtration
$F^i\Om^\bullet_{\calA/T}=\Imm(\Om^{\bullet-i}_{\calA/T}
\otimes_{\calO_\calA}\pi^*\Om^i_{S/T}\mmap\Om^\bullet_{\calA/T})$,
%with associated graded objects
%$\gr^i(\Om^\bullet_{\calA/T})=\Om^{\bullet-i}_{\calA/T}
%\otimes_{\calO_\calA}\pi^*\Om^i_{S/T}$,  
\cite{KatOda68}) as
\begin{equation}
\oom_{\calA/S}\hmap\derham{1}\stackrel{\nabla}{\mmap}\derham{1}
\otimes\Om^1_{S/T}\mmap\dual{(\oom_{\calA^t/S})}\otimes\Om^1_{S/T}.
\label{eq:KSfromGM}
\end{equation}
When $\calA_{/S}$ is principally polarized,
the Kodaira-Spencer map is a symmetric map
$\KS\colon\Sym^2(\oom)\map\Om^1_{S/T}$ under the identification 
$\calA_{/S}\simeq\calA^t_{/S}$, \cite[\S III.~9]{FalCha90}.

When $T=\Spec(\C)$ and $S$ is a complex variety the Kodaira-Spencer map of a 
family $\calA_{/S}$ can be readily computed, appealing to GAGA principles, 
working in the analytic category as in \cite{Katz76, Harris81}. By \'{e}tale-ness the 
actual computation can be carried out on the pullback of the family
$\calA_{/S}$ on the universal cover of $S$. If $\ze$ and $(\ze_1,\ze_2)$ denote the 
standard complex coordinate in the elliptic curve $E_z=\C/\Z\oplus\Z z$ and the
QM-abelian surface $A_z=\C^2/\Phi_\infty(\calR)\vvec{z}{1}$ respectively,
the following formulae for the Kodaira-Spencer map of the universal families
over $\gotH$ hold:
\begin{subequations}\label{eq:KSuniv}
\begin{eqnarray}
 \KS(d\ze^{\otimes 2}) & = & \frac{1}{2\pi i}dz,
\label{eq:KSuniv1} \\
\KS(d\ze_i\otimes d\ze_j)_{i,j=1,2}
	 & = & \frac1{2\pi i}\left(\begin{array}{cc}1 & 0 \\0 & \De\end{array}\right)\,dz
\label{eq:KSuniv2} 
\end{eqnarray}
\end{subequations}
Formula \eqref{eq:KSuniv1} is just the simplest case of the general formula 
for the universal family of principally polarized abelian varieties of genus 
$g$ (e.g. \cite[\S4.4]{Harris81}). Formula \eqref{eq:KSuniv2} is obtained in 
\cite[Prop. 2.3]{Mori11} in a similar fashion using the Hashimoto model
\cite{Hash95}.

When $D$ is split, the Kodaira-Spencer map attached to the universal 
elliptic curve over the Zariski open subscheme $\calY_\text{ar}(N)$, complement 
of the cusp divisor $C$ in $\Xar(N)$ is an isomorphism
$\KS\colon\oom^{\otimes2}\isom\Om^1_{\calY_\text{ar}(N)}$. It is well known that 
the line bundle $\oom$ extends uniquely at the cusps in the complete 
curve $\calX_1^1(N)$ and the Kodaira-Spencer isomorphism extends to an 
isomorphism
$$
\KS\colon\oom^{\otimes2}\iisom\Om^1_{\Xar(N)}(\log C)
$$
(see \cite{Katz73} and also \cite[\S 10.13]{KatMaz85} where the
extension property is discussed for a general representable
moduli problem). By flat base change
$$
H^0(\Xar(N),\oom^{\otimes k})\otimes_\Z\Z[1/N]\iisom
H^0(\calX_1^1(N)_{/\Z[1/N]},\oom^{\otimes k})
$$
and $f(z)\mapsto f^\ast=f(z)(2\pi i\,d\ze)^{\otimes k}$ defines an identification
$$
%\begin{equation}
%\label{eq:modformidentification}
M^1_{1,k}(N)\iisom H^0(X_1^1(N),\oom^{\otimes k})\simeq
H^0(\Xar(N),\oom^{\otimes k})\otimes\C
%\end{equation}
$$
normalized so that for a subring $B\subseteq\C$ 
the modular forms defined over $B$, i.e. the sections in
$H^0(\Xar(N)_{/B},\oom^{\otimes k})$ 
correspond to modular forms with Fourier coefficients in $B$ 
($q$-expansion principle, e.g. \cite{Katz73} \cite[theorem 4.8]{Harris81}).

If $D$ is not split, the Kodaira-Spencer map attached to the universal QM-abelian 
surface over $\XarD(N)$ has a big kernel. Let $p$ be a prime, $(p,\De)=1$, and
$v$ be a place in a number field $L$ dividing $p$.
The algebra $\calR\otimes_\Z\calO_{(v)}$ has a natural  $\calO_{(v)}$-linear action 
on $\oom$ over $\XarD(N)_{/\calO_{(v)}}$. If $e\in\calR\otimes_\Z\calO_{(v)}$ 
is a non-trivial idempotent, the subsheaf $e\oom$ is a line bundle
because for any geometric point $s\in\XarD(N)_{/\calO_{(v)}}$ 
the decomposition 
$$
H^0(A_s,\Om^1_{A_s/k(s)})=
eH^0(A_s,\Om^1_{A_s/k(s)})\oplus(1-e)H^0(A_s,\Om^1_{A_s/k(s)})
$$
is non-trivial. Then, \cite[Theorem 2.5]{Mori11}, the Kodaira-Spencer map restricts 
to an isomorphism of line bundles
$$
\KS\colon e\oom\circ\invol{e}\oom\iisom\Om^1_{\XarD(N)/\calO_{(v)}}.
$$

\begin{notat}
    \rm Denote $\calL$ either the  line bundle $\oom$
    on $\Xar(N)_{/\calO_{(v)}}$ or the line bundle 
    $e\oom$ on $\XarD(N)_{/\calO_{(v)}}$ 
    for some choice of idempotent $e\in\calR_{1}\otimes_\Z\calO_{(v)}$ 
    with $v$ as above and such that $\invol e=e$. In either 
    case the Kodaira-Spencer map gives an isomorphism
    $\KS\colon\calL^{\otimes 2}\stackrel{\sim}{\mmap}\Om^{1}$.
\end{notat}
    
The complex line bundle base change of $\calL^{\otimes k}$ to the Riemann surfaces 
$Y_1(N)$ or $X_1^\De(N)$ can be constructed also by taking the quotient
of the trivial complex line bundle $\calV_k=\gotH\times\C$,
homogeneous for the action $g\cdot(z,\ze)=(g\cdot z, j(g,z)^k\ze)$ of $\GL_2^+(\R)$,
by $\Ga_1^\De(N)$ respectively.
Fix a global constant non-zero section $v_k\in H^0(\gotH,\calV_k)$ and let 
$\varsigma(z)\in H^0(\gotH,\calL)$ such that
$\KS(\varsigma(z)^{\otimes 2})=2\pi i\,dz$
(the section $\varsigma(z)$ is given in \eqref{eq:KSuniv1} in the split case and 
can be computed explicitly in the non-split case, see \cite[Remark 2.6.3]{Mori11}).
The identifications over $\gotH$
\begin{equation}
	\calV_2\iisom\Om^1,
	\ v_2\mapsto 2\pi i\,dz\qquad
	\hbox{and}\qquad
	\calV_{k}\iisom{\calL}^{\otimes k},
	\ v_{k}\mapsto{\varsigma(z)}^{\otimes k}
	\label{eq:Vkbundles}
\end{equation}
descend to isomorphisms of complex line bundles on  $Y_1(N)$ or $X_1^\De(N)$
that preserve holomorphic sections and are compatible 
with tensor products and the Kodaira-Spencer isomorphisms.
Thus, mapping $f(z)\mapsto f^\ast=f(z)\varsigma(z)^{\otimes k}$ 
allows to identify
$$
M^\De_{1,k}(N)\iisom H^0(X_1^\De(N),\calL^{\otimes k})\simeq
H^0(\XarD(N)_{/\calO_{(v)}},\calL^{\otimes k})\otimes\C
$$
(note that there is a sign ambiguity in the $k$ odd case), and say that
the modular form $f$ is defined over a ring $B$ of definition for $\calL$ if
the corresponding section is in 
$H^0(\calX_1^\De(N)_{/B},\calL^{\otimes k})$.

Katz \cite{Katz76} pointed out that in a situation where the Hodge sequence \eqref{eq:Hodgeseq} 
admits a splitting ${\rm Pr}:\derham1\rightarrow\oom$ an operator
$\Th_{k,{\rm Pr}}:\Sym^k(\oom)\mmap\Sym^k(\oom)\otimes\calL$
can be defined as the composition
\begin{equation}
\label{eq:algmaassnsp}
\xymatrix{
\Sym^k(\oom) \ar@{^{(}->}[r]\ar@{-->}[ddrr]_{\Th_{k,{\rm Pr}}} & \Sym^k(\derham1)\ar[r]^-\nabla & 
\Sym^k(\derham1)\otimes\Om^1\ar[d]^{1\otimes\inv{\KS}} \\
 & & \Sym^k(\derham1)\otimes\calL^{\otimes 2} \ar[d]^{\Sym^k({\rm Pr})\otimes1} \\
 & & \Sym^k(\oom)\otimes\calL^{\otimes 2} 
}
\end{equation}
At least two such situations occur 
after a suitable modification of the basis $S=\XarD(N)$ with its universal family.
\begin{itemize}
  \item \textit{Classical (archimedean) Maass operators.} If $T=\Spec(\C)$, 
           consider the associated differentable manifold 
           $X_1^\De(N)^\text{an}$ and denote $\calM_\infty$ the sheaf of 
           $\calC^\infty(X_1^\De(N)^\text{an})$-modules associated to the algebraic sheaf 
           $\calM$. The Hodge decomposition 
           $(\derham1)_\infty=\oom_\infty\oplus\overline{\oom}_\infty$ defines a splitting
           ${\rm Pr}_{Hodge}:(\derham1)_\infty\rightarrow\oom_\infty$ and ultimately an operator
           $$
           \Th_{k,\infty}:\calL_\infty^{\otimes k}\mmap\calL_\infty^{\otimes k+2}
           $$
           (\cite[Prop. 2.8]{Mori11} for the details about the non-split case). 
           The operator $ \Th_{k,\infty}$ is the classical 
           one-dimensional Maass operator
           $$
           \de_k=\frac1{2\pi i}\left(\frac{d}{dz}+\frac{k}{2iy}\right),
           \qquad z=x+yi,
           $$
           in disguise, since under \eqref{eq:Vkbundles}
           there is a commutative diagram
	 $$
	 \xymatrix{
	 (\calV_k)_\infty\ar[r]^\sim\ar[d]_{\widetilde{\de}_k} & \calL_\infty^{\otimes k}\ar[d]^{\Th_{k,{\rm Hodge}}} \\
	 (\calV_{k+2})_\infty\ar[r]^\sim &  \calL_\infty^{\otimes k+2} 
	 }
	 $$
	where $\widetilde{\de}_k(\phi v_k)=\de_{k}(\phi)v_{k+2}$,
	\cite[Prop. 2.9]{Mori11}, i.e.
	\begin{equation}\label{eq:thetainf}
	 \Th_{k,\infty}\left(f(z)\varsigma(z)^{\otimes k}\right)=
	 \de_k(f)(z)\varsigma(z)^{\otimes k+2}.
	\end{equation}
           
  \item \textit{$p$-adic Maass operators.}
            Let $B$ be a $p$-adic algebra, $(p,\De)=1$, such that 
            $e$ is defined over $B$. Let $\calO^{[p]}$ be the structure sheaf
            of the (possibly formal) smooth scheme 
            $\calX^{[p]}=\varprojlim_n(\XarD(N)_{/B/p^nB})^{\text{$p$-ord}}$
            obtained taking out at the finite steps the non-ordinary points in characteristic $p$.
            Note that if $p\mid N$ then $\calX^{[p]}=\XarD(N)_{/B}$ since $\XarD(N)$ does not 
            contain non-ordinary points in this case.
            Denote $\calM^{[p]}$ the tensorization with $\calO^{[p]}$ of the restriction 
            to $\calX^{[p]}$ of a sheaf $\calM$ on $\calX$.
            
            The Dwork-Katz construction \cite[\S A2.3]{Katz73} of the unique 
            Frobenius-stable $\nabla$-horizontal submodule 
            $\calU\subset\derham{1}\otimes B$ can be carried out in the non split case 
            as well. It defines a splitting $(\derham{1})^{[p]}=\oom^{[p]}\oplus\calU$ 
            with projection
            $\PFr\colon(\derham{1})^{[p]}\map\oom^{[p]}$
            and ultimately an operator
            $$
            \Th_{k,p}:({\calL}^{\otimes k})^{[p]}\mmap({\calL}^{\otimes k+2})^{[p]}
            $$
            (see \cite{Mori11} for the details about the non-split case)
\end{itemize}
The operators $\Th_{k,\ast}$ can be iterated by increasing $k$: for all $r\geq1$ let
$\Th_{k,\ast}^{(r)}=\Th_{k+2r-2,\ast}\circ\cdots\circ\Th_{k,\ast}$.
The operators $\Th_{k,\ast}^{(r)}$ are algebraic over the CM locus in the 
following sense. Let $x\in\CM(\De,N,K,c)$ in some arithmetic $p$-ordinary test 
object, with corresponding abelian variety $A=A_x$. The choice of an invariant 
form $\om_o$ on  $A_x$ which generates either $H^0(A,\Om^1_{A/\calO_{(v)}})$
in the split case 
or $eH^0(A,\Om^1_{A/\calO_{(v)}})$ in the non-split case
identifies the algebraic   fiber ${\calL}(x)=x^\ast\calL$ with a copy of $\calO_{(v)}$. 
Then, for all $r\geq1$ the operator $\Th_{k,\ast}^{(r)}$ define a map
\begin{equation}
\label{eq:nerondiff}
\Th_{k,\ast}^{(r)}(x)\colon
H^0(\XarD(N)_{/\calO_{(v)}},{\calL}^{\otimes k})\mmap{\calL}^{\otimes k+2r}(x)
\simeq\calO_{(v)}{\om_o}^{\otimes k+2r},
\end{equation}
and in fact
\begin{equation}
\label{eq:equality}
\Th_{k,\infty}^{(r)}(x)(f^*)=\Th_{k,p}^{(r)}(x)(f^*)
\end{equation}
for a modular form $f$ defined over $\calO_{(v)}$.
As in \cite[theorem~2.4.5]{Katz78} this 
follows from the fact that both the Hodge decomposition of 
$\derham{1}(A_{/\C})$ and the Dwork-Katz decomposition of 
$\derham{1}(A_{/B})$ for a $p$-adic $\calO_{(v)}$-algebra $B$
can be obtained by a suitable tensoring from the eigenspace decomposition 
of the algebraic $\derham{1}(A_{/\calO_{(v)}})$ under the action 
of the complex multiplications in $K$. Also, \eqref{eq:equality} holds for all 
$r\geq1$ because
$\Th_{k,\ast}^{(r)}={\rm Pr}_\ast\left((1\otimes\inv\KS)\nabla\right)^r$,
since the kernels of ${\rm Pr}_\ast$ are $\nabla$-horizontal.
For all $r\geq1$ write
$\de_k^{(r)}=\de_{k+2r-2}\circ\cdots\circ\de_k$.

%%%%%%  CHAPTER TWO  %%%%%%%

\section{Expansions and measures}

\subsection{Serre-Tate Theory}
Let $\kk$ be a field of characteristic $p>0$, 
$\La$ a complete local ring with residue field $\kk$,
$\calC_\La$ the category of artinian local
$\La$-algebras with residue field $\kk$ and $\AA_{/\kk}$ an abelian
variety of dimension $g$. Grothendieck has shown that
the local moduli functor
$$
\calM_\AA\colon\calC_\La\map\mbox{\bf Sets},\qquad
\calM_\AA(B)=
\left\{
\sopra{\hbox{Abelian schemes $\calA=\calA_{/B}$}}
{\hbox{such that $\calA\otimes_\La\kk=\AA$}}
\right\}.
$$
is pro-represented by $\La[[t_1,\ldots,t_{g^2}]]$.
The Serre-Tate theory makes precise
the fact that a deformation of $\AA$ to $B\in\calC_\La$ is completely determined 
by a deformation of  its formal group.
When $\kk$ is algebraically closed, $\La=\calW$ is a finite extension 
of the ring $W_\kk$ of Witt vectors and $\AA$ is
\emph{ordinary}, all of which shall be assumed henceforth,
there is a canonical isomorphism of functors
$$
\calM_\AA\iisom\Hom(T_p\AA\otimes T_p\AA^t,\widehat{\G}_m),
$$
%\cite[theorem 2.1]{Katz81}
which  endows $\calM_\AA=\Spf(\gotR^u)$ with a canonical
structure of formal torus and identifies its group of characters
$X(\calM_\AA)=\Hom_{\Z_p}(\calM_\AA,\widehat{\G}_m)\subset\gotR^u$
with $T_p\AA\otimes T_p\AA^t$. 
For a deformation $\calA_{/B}$ of $\AA$, let
$$
q(\calA_{/B};\cdot,\cdot)\colon
T_p\AA\times T_p\AA^t\mmap\widehat{\G}_m(B)=1+\gotm_B
$$
be the corresponding bilinear form. Briefly, 
$q(\calA_{B};P,Q^t)=\langle\varphi_\calA(P),Q^t\rangle$ 
(essentially a Weil pairing)
where $\varphi_\calA:T_p(\AA)\rightarrow\widehat\calA$
is the limit of the Drinfeld lifts of the multiplication by $p^n$ maps
and realizes the canonical connected-\'etale exact sequence
for the $p$-divisible group of $\calA$ as the pushout of the standard exact 
sequence for $T_p(\AA)$,
$$
\xymatrix{
0\ar[r] & \widehat\calA\ar[r] & \calA[p^\infty]\ar[r] & T_p(\AA)\otimes_{\Z_p}(\Q_p/\Z_p)\ar[r] & 0 \\
0\ar[r] & T_p(\AA)\ar[r]\ar[u]_{\varphi_{\calA}} & T_p(\AA)\otimes_{\Z_p}\Q_p\ar[r]\ar[u] & 
               T_p(\AA)\otimes_{\Z_p}(\Q_p/\Z_p)\ar[r]\ar@{=}[u] & 0
}
$$
Grothendieck's result is thus explicit, in the sense that
if $T_p(\AA)=\bigoplus_{i=0}^g\Z_pP_i$ and $T_p(\AA^t)=\bigoplus_{j=0}^g\Z_pP_j^t$,
and $\calA^u_{/\gotR^u}$ is the universal formal deformation,
the $g^2$ elements
$$
q_{i,j}=q(\calA^u_{/\gotR^u};P_i,P_j^t)-1=
\varprojlim_{B\in\cal C}q(\calA_{/B};P_i,P_j^t)-1
$$
define an isomorphism
$\gotR^u\simeq\calW[[q_{i,j}]]$.
The following two facts (for, and for the details of all of the above as well, 
see \cite{Katz81}) are crucial.
\begin{rem}\label{fa:one}\rm
If $\calA_{/B}$ and $\calB_{/B}$ are formal deformations of the ordinary abelian varieties 
$\AA$ and $\BB$ respectively, a morphism $f\colon\AA\map\BB$ lifts to a morphism
$f_B\colon\calA\map\calB$ if and only if
\begin{equation}
\label{eq:lift}
q(\calA_{/B};P,f^t(Q^t))=q(\calB_{/B};f(P),Q^t)\quad
\mbox{for all $P\in T_p\AA, Q^t\in T_p\BB^t$}.
\end{equation}
In particular, if $\AA$ is principally polarized, under the  identification $T_p\AA\simeq T_p\AA^t$ 
the formal subscheme $\calM^{\rm pp}_\AA$ that classifies deformations of $\AA$ with a lifting of 
the principal polarization is a subtorus with group of characters $X(\calM^{\rm pp}_\AA)=\Sym^2(T_p\AA)$.
\end{rem}

\begin{rem}\label{fa:two}\rm
The Weil pairing yields an isomorphism $T_p\AA^t\simeq\Hom_B({\widehat{\calA}},\widehat{\G}_m)$, hence
a functorial $\Z_p$-linear homomorphism $\om_\calA\colon T_p\AA^t\map\oom_{\calA/B}$,
where $\om_\calA(P^t)$ is the pullback of $dt/t$ under the morphism corresponding to $P^t$. 
Thus, if $\AA$ is principal\-ly polarized and $\calA_{/B}$ is a 
principally polarized deformation
$$
\om_\calA(P)\cdot\KS(\om_\calA(Q))=d\log q(\calA_{/B};P,Q)\quad
\mbox{for all $P,Q\in T_p\AA$}
$$
under the usual identifications.
\end{rem}

By ordinarity, $\AA[p]=\widehat\AA[p]\times(T_p(\AA)\otimes_{\Z_p}\F_p)$.
The two factors are mutually Cartier dual finite group schemes, each of them giving 
rise to a quotient isogeny of degree $p^g$.
The connected factor is the kernel of the completely inseparable $\kk$-linear 
Frobenius morphism $F_\kk:\AA\map\AA^{(p)}$.
%where 
%$\AA^{(p)}$ is the abelian variety over $\kk$ obtained by twisting 
%the $\kk$-structure of $\AA$ by the geometric Frobenius $\al\mapsto\al^{1/p}$.
Denoting $\AA^{(1/p)}$ the quotient of $\AA$ by the \'etale factor, the separable isogeny
$G\colon\AA\rightarrow\AA^{(1/p)}$ is the Verschiebung for $\AA^{(1/p)}$
because $(\AA^{(1/p)})^{(p)}\simeq\AA$, i.e. $G^{(p)}=V_\kk$.

\begin{pro}\label{teo:qunderfrob}
      Let $\AA$ be endowed with a polarization $\la$ of degree not divisible by $p$. 
      Then there are identifications 
      $T_p(\AA)\simeq T_p(\AA^t)\simeq T_p(\AA^{(p)})\simeq T_p((\AA^{(p)})^t)
      \simeq T_p(\AA^{(1/p)})\simeq T_p((\AA^{(1/p)})^t)$ 
      under which a morphism $\phi:\calA\rightarrow\calB$
      of abelian schemes over $B\in\widehat{\cal C}$ lifts the Frobenius morphism $F_\kk$
      if and only if
      $$
      q(\calB;P,Q)=q(\calA;P,Q)^p\qquad\hbox{for all $P, Q\in T_p(\AA)$},
      $$
      and a morphism $\psi:\calA\rightarrow\calB$
      of abelian schemes over $B\in\widehat{\cal C}$ lifts $G$
      if and only if
      $$
      q(\calB;P,Q)^p=q(\calA;P,Q)\qquad\hbox{for all $P, Q\in T_p(\AA)$},
      $$
\end{pro}

\begin{proof}
      Since $\ker\la$ meets $\AA[p]$ trivially, there are canonically induced polarizations 
      $\la^{(p)}$ on $\AA^{(p)}$ and $\la^{(1/p)}$ on $\AA^{(1/p)}$ such that 
      $\deg(\la)=\deg(\la^{(p)})=\deg(\la^{(1/p)})$.
      Thus, the situation defines a (non-commutative!) diagram of Tate modules
      $$
      \xymatrix{
      T_p(\AA^{(1/p)})\ar[d]_\simeq^{\la^{(1/p)}} & T_p(\AA)\ar[d]_\simeq^{\la}\ar[r] _\sim^{F_\kk}\ar@{.>}[l]_G & 
      T_p(\AA^{(p)})\ar[d]_\simeq^{\la^{(p)}}  \\
      T_p((\AA^\si)^t)\ar[r]^\sim_{G^t} & T_p(\AA^t) & T_p((\AA^{(p)})^t)\ar@{.>}[l]_{F_\kk^t}
      }
      $$
      where the continuous arrows are isomorphisms ($G^t$ is a 
      $\kk$-linear Frobenius, as observed above) and the dotted arrows are, under these identifications,
      multiplication by $p$ maps.
      The result follows directly from \eqref{eq:lift}.
\end{proof}

If $\EE_{/\kk}$ is an elliptic curve, the local moduli functor $\calM^1_\EE=\Spf(\gotR^u)$
is a one-dimensional torus with $\gotR^u=\calW[[q-1]]$ where $q=q(\calE^u_{/\gotR^u};P,P)$ and
$T_p(\EE)=\Z_pP$. Also, $\KS(\om_u^{\otimes 2})=d\log(q)$ 
where $\om_u=\om_{\calE^u}(P)$.

For a QM-abelian surface $\AA_{/\kk}$ a choice of idempotent 
$e\in\calR\otimes\Z_p\simeq\M_2(\Z_p)$ defines a decomposition 
$T_p(\AA)=\Z_pP\oplus\Z_pQ$ of Tate modules where $eP=P$ and $eQ=0$. 
Since QM-abelian surfaces have a canonical principal polarization
it is natural to consider the subfunctor
$$
\calM_\AA^\De(B)=
\left\{
\sopra{\hbox{principally polarized deformations $\calA_{/B}$ of $\AA$ 
with a lift of the}}
{\hbox{endomorphisms given by elements of the
maximal order $\calR$}}
\right\}.
$$
Then, \cite[Proposition 3.3]{Mori11},
\begin{pro}\label{teo:shimuramoduli}
      The subfunctor $\calM_\AA^\De=\Spf(\gotR^u_D)$ is a $1$-dimensional subtorus 
      of $\calM_\AA$. Moreover, if $e\in\calR\otimes\Z_p$ is an idempotent such that 
      $\invol e=e$ and $\{P,Q\}$ is a $\Z_p$-basis of $T_p(\AA)$ such that 
      $eP=P$ and $eQ=0$, then
      \begin{enumerate}
  \item $\Spf(\gotR^u_D)=\calW[[q-1]]$, where 
            $q=\left.q(\calA^u;P,P)\right|_{\calM_\AA^\De}$;
  \item $\KS(\om_u^{\otimes 2})=d\log q$, where 
            $\om_u=\left.\om_{\calA^u}(P)\right|_{\calM_\AA^\De}$.
\end{enumerate}
\end{pro}

\begin{rem}
      \rm Let $\vep\in\calR\otimes\Z_p$ such that $\vep P=Q$, $\vep Q = P$,
      and let $\calA\in\calM_\AA^\De(B)$.
      Then, by continuity of the action of $\calR$ on $T_p(\AA)$,
      $$
      \left.q(\calA;P,Q)\right|_{\calM_\AA^\De}=
      \left.q(\calA;eP,Q)\right|_{\calM_\AA^\De}=
      \left.q(\calA;P,eQ)\right|_{\calM_\AA^\De}=1
      $$
      and since $\invol\vep$ acts on $\{P,Q\}$ as $\smallmat0\rho{1/\rho}0$ 
      for some $\rho\in\Z_p^\times$,
      $$
      \left.q(\calA;Q,Q)\right|_{\calM_\AA^\De}=
      \left.q(\calA;\vep P,Q)\right|_{\calM_\AA^\De}=
      \left.q(\calA;P,\invol\vep Q)\right|_{\calM_\AA^\De}=
      \left(\left.q(\calA;P,P)\right|_{\calM_\AA^\De}\right)^\rho.
      $$
      Thus for a morphism $f\colon\AA_1\map\AA_2$ of QM-abelian surfaces 
      to lift to a morphism of QM deformations is enough that 
      $q(\calA_2;f(P_1),P_2)=q(\calA_1;P_1,f^t(P_2))$ where $P_1\in T_p(\AA_1)$ and
      $P_2\in T_p(\AA_2)$ are as in proposition \ref{teo:shimuramoduli}.
\end{rem}

Let $\AA_{/\kk}$ be either an elliptic curve or a QM-abelian surface and 
$P\in T_p(\AA)$ either a $\Z_p$-generator or as in proposition \ref{teo:shimuramoduli} respectively. 
Let $\calA\rightarrow\calB$ a lift of the $\kk$-linear
Frobenius  over $B\in\widehat{\calC}$ and let $q_0=q(\calA;P,P)$.

\begin{pro}\label{teo:bijecofdefs}
    Let $\De\geq1$. The following two sets are in bijection:
    \begin{enumerate}
         \item the formal deformations $\calA^\prime\in\calM_\AA^\De(B)$ 
                   such that $q(\calA^\prime;P,P)\in q_0\mu_p(B)$;
         \item the \'etale subgroups $C$ of $\calB[p]$ of rank $p^{\dim(\AA)}$ which are
                   $\calR$-invariant when $D$ is non split.
%         \item the closed subgroups $C$ of $\calB[p]$ locally free of rank $p^{\dim\AA}$ which are
%                   $\calR$-invariant when $D$ is non split and such that 
%                   $C\otimes_B\kk=T_p(\AA^{(p)})\otimes(\Z_p/p\Z_p)$.
    \end{enumerate}
\end{pro}

\begin{proof}
An immediate consequence of proposition \ref{teo:qunderfrob} is that up to the identification
$(\AA^{(1/p)})^{(p)}\simeq\AA$ the first set is the set of deformations 
$\calA^\prime\in\calM_\AA^\De$ 
for which there is a lift $\calA^\prime\rightarrow\calB$ of the $\kk$-linear Frobenius for $\AA$ 
and also those for which there is a lift $\calB\rightarrow\calA^\prime$ of 
$V_\kk=G^{(p)}:\AA^{(p)}\rightarrow\AA$. 
On the other hand, if $\psi:\calB\rightarrow\calA^\prime$ lifts $V_\kk$, the subgroup 
$ker(\psi)<\calB[p]$ is in the second set. Finally, if $C$ is in the second set, the quotient $\calB/C$ 
is a deformation of $\AA$ and the quotient map lifts $V_\kk$.
\end{proof}

The bijection can be made explicit as follows. 
The deformation $\calB$ can be recovered from the universal deformation $\calB^u$ of $\AA^{(p)}$
as pullback via the classifying map $\phi_\calB:\calW[[q-1]]\map B$ such that $\phi_\calB(q-1)=q_0^p-1$.
Since $\phi_\calB$ factors through the map $\al:\calW[[q-1]]\map\calW[[q-1]]$ such that $\al(q)=q^p$
there is a diagram of pullbacks
$$
\xymatrix{
\calB\ar[r]\ar[d] & \calB^\prime\ar[r]\ar[d] & \calB^u\ar[d] \\
\Spec(B)\ar[r]\ar@/_1pc/[rr]_{\phi_\calB}& \calM_{\AA^{(p)}}\ar[r]^\al & \calM_{\AA^{(p)}}
}
$$
and a map $\calA^u\map\calB^\prime$ which is a lift of $F_\kk$.
Let $C<\calB[p]$ as in the second set of proposition \ref{teo:bijecofdefs}.
Let $\widetilde C$ an \'etale subgroup of $\calB^u$ such that 
$C=\widetilde C\times_{\phi_\calB}B$
and let $C^\prime=\widetilde C\times_\al\calM_{\AA^{(p)}}$.
Also, let $\calB^\can_{/B}$ be the pullback of $\calB^u$ (or $\calB^\prime$) via the map $q\mapsto1$ 
and $C^\can$ the pullback of $\widetilde C$ (or $C^\prime$). The deformation $\calB^\can$ 
is the canonical lift of $\AA^{(p)}$, characterized by $q(\calB^\can;P,P)=1$ or,
equivalently by the splitting of the connected-\'etale exact sequence of its $p$-divisible group,
i.e. $\calB^\can[p^\infty]=\widehat\calB^\can\times T_p(\AA^{(p)})\otimes(\Q_p/\Z_p)$.
Up to a choice of isomorphism $\widehat\calB^\can[p]=\mu_{p}(B)^{\dim\AA}$
(where for QM-abelian surfaces the factors are chosen so that the action of $e$ is projection 
onto the first) write $C^\can=\langle(\ze,P_1)\rangle$ for some $\ze\in\mu_p(B)$ 
(such that $e\ze=\ze$ when $\dim(\AA)=2$), where $P=\varprojlim P_n$. Since the quotient 
isogeny $\calB^\can\map\calB^\can/C^\can$ induces an isomorphism of formal groups, the
 connected-\'etale exact sequence for $(\calB^\can/C^\can)[p^\infty]$,
 $$
 0\map\widehat{\calB^\can/C^\can}\mmap
 \frac{\widehat\calB^\can\times T_p(\AA^{(p)})\otimes(\Q_p/\Z_p)}{\langle(\ze,P_1)\rangle}
 \mmap
 T_p(\AA)\otimes(\Q_p/\Z_p)\map0,
 $$
arises as pushout for the map $\varphi_{\calB^\can/C^\can}(P)=\inv\zeta$. Thus, 
$$
\frac{q(\calB/C;P,P)}{q(\calA;P,P)}=
\frac{q(\calB^\can/C^\can;P,P)}{q(\calA^\can;P,P)}=\inv\ze
$$
by the connectedness of the local moduli space.

\subsection{Expansions}
Let $(p,M)=1$ and $\calT_\nu=(x_\nu,v_\nu,e)$, $\nu\geq0$, an arithmetic $p$-ordinary (split if $\De>1$) object
for $(\De,Mp^\nu,K)$. The abelian variety $A$ corresponding to $x_\nu$ has a smooth model $A_v$ 
over $\calO_{(v)}$ with $v\mid v_\nu$ equipped with an arithmetic level structure $\eta_{Mp^\nu}$ of level $Mp^\nu$.
Let $\kk=\overline{k}_v$ and $\AA=A_v\otimes\kk$ the reduction of $A$ modulo $v$ base changed to $\kk$.
Let $\calW\subseteq\calW_\nu$ be finite extensions of $W_\kk$ containing
$\calO_v$ and $\calO_{v_\nu}$ respectively. 
By smoothness, the local moduli space $\calM^\De_\AA(\calW_\nu)=\Spf(\gotR^u\otimes\calW_\nu)$ 
is the formal completion of
$\XarD(Mp^\nu)_{/\calW_\nu}$ along $x_\nu$ and there is a canonical isomorphism
$$
\gotR^u\otimes\calW_\nu\simeq\widehat\calO_{\XarD(Mp^\nu)_{/\calW_\nu},x_\nu},
$$ 
limit of the classifying maps $\gotR^u\otimes\calW_\nu\map\calO_{x_\nu}/\gotm_{x_\nu}^n$ 
that correspond to the restriction of the universal
family to $\Spec(\calO_{x_\nu}/\gotm_{x_\nu}^n)$. This identification allows to extend the arithmetic
structure $\eta_{Mp^\nu}$ to $\calA^u_{/\gotR^u\otimes\calW_\nu}$.  
Fix $P\in T_p(\AA)$ as in
proposition \ref{teo:shimuramoduli} and the discussion preceding it and let $q=q(\calA^u;P,P)$ and 
$\oom_u=\oom_{\calA^u}(P)$. Letting $q_0=q(A_v)\in1+\gotp_\calW$, it is also
\begin{equation}
\label{eq:locring}
\widehat\calO_{\XarD(Mp^\nu)_{/\calW_\nu},x_\nu}\simeq\calW_\nu[[q_x]]
\qquad\text{where $q_x=q_0^{-1}q-1$.}
\end{equation}
The parameter $q_x$ is a canonical local formal parameter \lq\lq centered at $q_0$\rq\rq. It is canonical 
in the sense that it is the pullback of the Serre-Tate parameter $q-1$ under the translation by $\inv q_0$ 
in the formal torus $\calM^\De_\AA$.

Let $f\in M_{1,k}^\De(Mp^\nu;\calO_{(v_\nu)})$ a modular form defined over $\calO_{(v_\nu)}$. 
The corresponding section $f^\ast\in H^0(\calM^\De_\AA,\calL^{\otimes k})$ at $x_\nu$
can be written $f^\ast=f_u\om_u^{\otimes k}$ with $f_u\in\gotR^u\otimes\calW_\nu$.
Thus, the identification \eqref{eq:locring} allows to write an expansion
\begin{equation}\label{eq:STexp}
f_u=F_f(q_x)=\sum_{n=0}^\infty a_n(f,x_\nu)q_x^n.
\end{equation}
We shall also write
\begin{equation}\label{eq:STlogexp}
f_u=\Phi_f(Q_x)=\sum_{n=0}^\infty\frac{b_n(f,x_\nu)}{n!}Q_x^n=F_f(e^{Q_x}-1)
\end{equation}
under the formal substitution $Q_x=\log(1+q_x)=q_x-\frac12q_x^2+\frac13q_x^3-\cdots$.

\begin{rem}\rm
It turns out that $b_n(f,x_\nu)\in\calW_\nu$ for all $n\geq0$. 
Moreover, if the coefficients $\ga_{n,i}\in\Z$ 
are defined by the identity $n!\binom Xn=\sum_{i=0}^n\ga_{n,i}X^i$, then
$$
\frac1{n!}\sum_{i=1}^n\ga_{n,i}b_i(f,x_\nu)\in\calW_\nu
\qquad\text{for all $n\geq1$}.
$$ 
\end{rem}

\begin{rem}\label{rem:faspadic}\rm
The choice of $P\in T_p(\AA)$ induces also a trivialization
$$
\varphi^{-1}:\hat{\G}_m^{\dim(A)}=\varinjlim_{j>0}\mu_{p^j}^{\dim(A)}\iisom
\hat{A}
$$
i.e. a compatible sequence of arithmetic structures $\eta_{p^\nu}$ on $A$ which together with a structure 
$\eta_M$ of level $M$ gives rise to a sequence $\calT_j=(x_j,v_j,e)$, $j\geq0$,  
of arithmetic $p$-ordinary (split if $\De>1$) objects for 
$(\De,Mp^j,K)$ which are compatible for the canonical forgetful maps  $\XarD(Mp^{j+1})\map\XarD(Mp^j)$.
The expansion \eqref{eq:STexp} remains the same if the object $\calT_\nu$ is replaced with $\calT_{\nu^\prime}$ 
for all $\nu\leq\nu^\prime$. Let $\calW_\infty=\varinjlim\calW_j$ and let 
$\varphi^u:\widehat{\calA}^u\iisom\widehat\G_m(\gotR^u\otimes\calW_\infty)^{\dim A}$ and $\eta_M^u$
be the unique extension of $\varphi$ and $\eta_M$ to the formal universal deformation respectively.
Then we can write
$$
F_f(q_x)=f(\calA^u,\varphi^u,\eta_M^u)
$$
where $f$ is reinterpreted as a generalized $p$-adic modular function in the sense of Katz
(for the extension of Katz's theory to quaternionic forms, see \cite{Kass99}).
%Here compatible means that $x_{\nu+1}$ maps to $x_\nu$
%under the natural forgetful map $\XarD(Mp^{\nu+1})\map\XarD(Mp^\nu)$.
%Spelling this out, the underlying abelian scheme $A_v$ is the same for both points and
%$$
%\xymatrix{
%\mu_{Mp^j}^{\dim(A)}\ar[r]\ar@/_1pc/[rrr]_{\eta_{Mp^j}} & \mu_{Mp^{j+1}}^{\dim(A)}\ar[rr]^{\eta_{Mp^{j+1}}} & 
% & A_v
%}.
%$$
%commutes for the canonical inclusion. In fact, the $p$-part $\eta_{p^\nu}$ of the arithmetic level structure $\eta_{Mp^\nu}$ 
%can always be extended to a trivialization
%$$
%\varphi_x^{-1}:\hat{\G}_m^{\dim(A)}(\calW_\infty)=\varinjlim_{j>0}\mu_{p^j}^{\dim(A)}(\calW_j)\iisom
%\hat{A}_{/\calW_\infty}
%$$
%where $\calW_\infty=\varinjlim\calW_j$. Moreover, the point $P\in T_p(\AA)$ can always be chosen so that
%the corresponding trivialization (remark \ref{fa:two}) coincides with $\varphi_x$. Thus, if 
%$\varphi^u:\widehat{\calA}^u\iisom\widehat\G_m(\gotR^u\otimes\calW_\infty)^{\dim A}$
%is the unique extension of $\varphi_x$ to the formal universal deformation, we can write
%$$
%F_f(q_x)=f(\calA^u,\varphi^u,\eta_M^u)
%$$
%where $f$ is reinterpreted as a generalized $p$-adic modular function in the sense of Katz \cite[\S I.3]{Go88}
%(for the extension of Katz's theory to quaternionic forms, see \cite{Kass99}) and $\eta^u_M$ is the constant prime-to-$p$
%part of the various level structures uniquely extended to the universal formal deformation.
\end{rem}

To a choice of an invariant form $\om_o$ on $A_v$ normalized as in \eqref{eq:nerondiff}
are attached periods as follows.

\begin{itemize}
\item \textit{Archimedian period.} Fix a complex uniformization of $A_{/\C}$, which amounts to
          choosing a suitable $\tau\in\gotH$. For a given $\om\in H^0(A_\tau,\calL(\tau))$
          write $\om=\perr(\om,\tau)\varsigma(\tau)$ with $\perr(\om,\tau)\in\C $ and let
          $$
          \Om_\infty=\perr(\om_o,\tau).
          $$
\item \textit{$p$-adic period.} For $\om\in H^0(A_{/\calW},\calL)$ write
         $\om=\per_p(\om,P)\om_u(P)_{|A}$  with $\per_p(\om,P)\in\calW$ and let
         $$
         \Om_p=\per_p(\om_o,P)\in\calW^\times.
         $$ 
\end{itemize}

In turn, the archimedean period can be used to define numbers
$$
\th_v^{(r)}(f,x)=\th^{(r)}(f,x,\om_o)\in\C
$$
where $\th^{(r)}(f,x_\nu,\om)={\perr(\om,\tau)^{-k-2r}}{\de_k^{(r)}(f)(\tau)}$ for any
$\om\in H^0(A_\tau,\calL(\tau))$.
These numbers depend only on $x$ and not on the choice of a representant $\tau$.
The following result relates the periods to the expansion \eqref{eq:STexp}.

\begin{thm}\label{integcoeff}
Let $f\in M_{1,k}^\De(Mp^\nu;\calO_{(v_\nu)})$ and $x\in\XarD(Mp^\nu)(\calO_{(v_\nu)})$ 
belonging to a $p$-ordinary arithmetic test object (split, if $\De>1$). Then
$$
\th_v^{(r)}(f,x)=\Om_p^{-k-2r}b_r(f,x)\in\calO_{(v_\nu)}\qquad\text{for all $r\geq0$}.
$$
\end{thm}

\begin{proof}
The result is obvious for $r=0$, so that  $r\geq1$ may be assumed. Since $\PFr(\nabla(\om_u))=0$,
a consequence of  \cite[theorem 4.3.1]{Katz81} is that
$\Th_{k,p}(f)=(df_u/dQ_x)\om_u^{\otimes k+2}$. After $r$ iterations and evaluation at $x$,
$$
\Th_{k,p}^{(r)}(f)(x)=
\frac{d^rf_u}{d{Q_x}^r}(x)\om_u^{k+2r}(x)=
\frac{b_r(f,x)}{\Om_p^{k+2r}}\om_o^{k+2r}.
$$
On the other hand, \eqref{eq:thetainf} yields
$$
\Th_{k,\infty}^{(r)}(f)(x)=
\de_k^{(r)}(f)(\tau)s(\tau)^{\otimes k+2r}=
\th_v^{(r)}(f,x)\om_o^{k+2r}.
$$
The result follows from \eqref{eq:equality}.
\end{proof}

By the classical theory of Mahler 
%the coefficients of a power series in $\calW_\nu[[T]]$
%encode a $\calW_\nu$-valued measure on $\Z_p$. 
the series $F_f$ in \eqref{eq:STexp}, defines a a $\calW_\nu$-valued measure
$\mu_{f,x_\nu}$ on $\Z_p$ characterized by
$\int_{\Z_p}\binom tr\,d\mu_{f,x_\nu}(t)=a_r(f,x_\nu)$ for all $r\geq0$.
Also, the moments of $\mu_{f,x_\nu}$ can be read off the expansion  \eqref{eq:STlogexp}, namely
$$
m_r(\mu_{f,x_\nu})=\int_{\Z_p}t^r\,d\mu_{f,x_\nu}(t)=b_r(f,x_\nu)=\Om_p^{k+2r}\th^{(r)}_v(f,x_\nu)\quad
\text{for all $r\geq0$}.
$$
\begin{rem}\rm
By theorem \ref{integcoeff} the measure $\mu_{f,x_\nu}$ takes values in the field $L_{v_\nu}(\Om_p)$,
which is well-defined since $\Om_p$ is well-defined as an element of $\calW^\times/\calO_v^\times$.
\end{rem}

Denote $\mu^\times$ the restriction of the measure $\mu$ to $\Z_p^\times$, i.e.
$\int_{\Z_p^\times}\phi\,d\mu^\times=\int_{\Z_p}\phi^\ast\,d\mu$
where $\phi^\ast$ is the function obtained from $\phi$ setting it equal to $0$ on $p\Z_p$.
%for every continuous function $\phi$ on $\Z_p$. 
The following theorem shows how the operation of restricting 
$\mu_{f,x_\nu}$ to $\Z_p^\times$ is strongly related to the action of the $p$-th Hecke operator on $f$.
The operators $U$ and $V$ in the statement denote the classical  operators 
(for which the explicit dependence on $p$ will be dropped).

\begin{thm}\label{th:mustar}
Let $f\in M_{1,k}^\De(Mp^\nu;\calO_{(v_\nu)})$.
%and $x\in\XarD(N)(\calO_{(v)})$ belonging to a $p$-ordinary
%arithmetic test object (split, if $\De>1$). 
Then
$\mu^\times_{f,x_\nu}=\mu_{(1-VU)f,x_{\nu+1}}$.
\end{thm}

\begin{proof}
The general theory says that the formal power series encoding $\mu_{f,x_\nu}^\times$ is
$$
F^\times_f(T)=F_f(T)-\frac1p\sum_{\ze\in\mu_p}F_f(\ze T+\ze-1).
$$
Following remark \ref{rem:faspadic} we read the summands in the right hand side (after evaluation at $T=q_x$)
as values of $f$ at particular trivialized abelian schemes over $\calW_\infty[[q_x]]$. Given a map of 
$\calW_\infty$-algebras $h:\calW_\infty[[q_x]]\map B$ there is an equality
$h(F_f(q_x))=f(\calA^h,\varphi^h,\eta_\nu^h)$
where the triple $(\calA^h,\varphi^h,\eta_M^h)$ is obtained from 
$(\calA^u,\varphi^u,\eta_M^u)$ by pull-back:
%\begin{equation}\label{eq:pb}
%\xymatrix{
% & \calA^h\ar[rr]\ar[dd] & & \calA^u\ar[dd] \\
%(\widehat{G}_m\times\mu_{M})^{\dim(A)}{}_{/B}\ar[rr]\ar[dr]\ar[ur]^{(\varphi^h)^{-1}\times\eta_M} & & 
%(\widehat{G}_m\times\mu_{M})^{\dim(A)}{}_{/\calW_\infty[[q_x]]}\ar[dr]\ar[ur]^{(\varphi^u)^{-1}\times\eta_M}  \\
% & \text{Spf}(B)\ar[rr]^{h^*} &  & \calM_\AA^\De
% }
%\end{equation}
\begin{equation}\label{eq:pb}
\xymatrix{
(\widehat{\G}_m\times\mu_{M})^{\dim(A)}{}_{/B}\ar[r]\ar[d]^{(\varphi^h)^{-1}\times\eta_M} & 
(\widehat{\G}_m\times\mu_{M})^{\dim(A)}{}_{/\calW_\infty[[q_x]]}\ar[d]^{(\varphi^u)^{-1}\times\eta_M}\\
  \calA^h\ar[r]\ar[d] & \calA^u\ar[d] \\
%(\widehat{G}_m\times\mu_{M})^{\dim(A)}{}_{/B}\ar[rr]\ar[dr]\ar[ur]^{(\varphi^h)^{-1}\times\eta_M} & & 
%&(\widehat{G}_m\times\mu_{M})^{\dim(A)}{}_{/\calW_\infty[[q_x]]}\ar[dr]\ar[ur]^{(\varphi^u)^{-1}\times\eta_M}  \\
 \text{Spf}(B)\ar[r]^{h^*} &  \calM_\AA^\De
 }
\end{equation}
Thus,
$$
F^\times_f(q_x)=f(\calA^u,\phi^u,\eta_M^u)-
\frac1p\sum_{\ze\in\mu_p}f(\calA^\ze,\phi^\ze,\eta_M^\ze)
$$
where the triple $(\calA^\ze,\phi^\ze,\eta_\nu^\ze)$ is obtained as above from the endomorphism 
of $\calW_\infty[[q_x]]$ determined by $q_x\mapsto\ze q_x+\ze-1$.
Since $\ze q_x+\ze-1=q_0^{-1}(\ze q)-1$, by proposition \ref{teo:bijecofdefs} the sum extends 
over the set ${\cal F}$ of deformations $\calA_{/\calW_\infty[[q_x]]}$ admitting a lift 
$\pi_\calA:\calA\map\calA^{u,p}=\calA^u/H^u$ of the $\kk$-linear Frobenius
($H^u<\calA^u$ the canonical subgroup). 
On the other hand, if $U_p$ and $V_p$ denote the $p$-adic counterparts of the 
classic $U$ and $V$ operators (e.g. \cite{Go88}),
$$
V_pU_pf(\calA^u,\varphi^u,\eta_M^u)=
U_pf(\calA^{u,p},\varphi^{u,p},\eta_M^{u,p})=
\frac1p\sum_{\calA\in\cal F}
f(\calA,\phi^\calA,\eta_M^\calA)
$$
where $\varphi^{u,p}=\varphi^u\pi_{\calA^u}^t$, $\eta_M^{u,p}=\pi_{\calA^u}\mu_M^u$ and for each $\calA\in\calF$, 
$\varphi^\calA=\varphi^u\pi_{A^u}^t(\pi_{\calA}^t)^{-1}$ and $\eta_M^\calA=\pi_{\calA}^{-1}\pi_{\calA^u}\eta_M^u$
(the notation is justified because each lift of Frobenius is an isomorphism between $M$-torsion subgroups and its dual
is \`etale and an isomorphism between formal groups).
Since $\varphi^\calA=\varphi^\ze$ and  $\eta_\nu^\calA=\eta_\nu^\ze$ when
$\calA=\calA^\ze$ because the lifts of Frobenius commute with the pull-back map \eqref{eq:pb} this proves that
$\mu^\times_{f,x_\nu}=\mu_{(1-V_pU_p)f,x_\nu}$.

To finish the proof, observe that for $f\in M_{1,k}^\De(Mp^\nu;\calO_{(v_\nu)})$
and for every triple $(\calA,\varphi,\eta_M)$ defined over a $p$-adic ring
$U_pf(\calA,\varphi,\eta_M)=Uf(\calA,\eta_M\times\varphi_{\left|\mu_{p^\nu}\right.}^{-1},\om_\calA(\varphi))$ and
$V_pf(\calA,\varphi,\eta_M)=Vf(\calA,\eta_M\times\varphi_{\left|\mu_{p^\nu}\right.}^{-1},\om_\calA(\varphi))$ 
 because of the compatibility of the $p$-part of the level structures, where $\om_\calA(\varphi))$ is the unique extension
 to $\calA$ of the invariant $1$-form defined by the trivialization (see remarks \ref{fa:two} and \ref{rem:faspadic}).
\end{proof}

\subsection{More measures}\label{se:moremeas}
A weight for the imaginary quadratic imaginary field $K$ is a formal linear combination
$\uw=(w_1,w_\si)=w_1\text{id}+w_\si\si\in\Z[\text{Gal}(K/\Q)]$. Write $\la^\uw=\la^{w_1}\si(\la)^{w_\si}$
for all $\la\in K$.
For an ideal $\gotn\subset\calO_K$, a weight $\uw$ and a field $K\subseteq E$ let
$$
H_\uw(\gotn;E)=
\left\{
\text{$\phi:\calI_\gotn\map E$ such that $\phi((\la)I)=\la^\uw\phi(I)$ for all $\la\in K_\gotn$}
\right\}
$$
(Hida's $K^\times$-modular forms of level $\gotn$, weight $\uw$ and coefficients in $E$, \cite{Hida86}). 
From an observation of Weil in \cite{Weil55} one knows that the algebraic Hecke characters in $H_\uw(\gotn;\C)$
are in $H_\uw(\gotn;\overline\Q)$. When $E\subseteq\overline\Q$, there are 
local avatars for the spaces $H_\uw(\gotn;E)$.
Denote $\gotn^\ast$ the prime to $p$ part of $\gotn$.
\begin{itemize}
  \item \textit{Archimedean avatars.}                 
               Denote $\arH_\uw(\gotn)$ the space of $\C$-valued functions on $K_\A^\times$ such that
               \begin{equation}
               \label{eq:grossen}
               \Phi(t\la zu)=z^{-\uw}\Phi(t)
               \qquad\text{for all $\la\in K^\times$, $z\in\C^\times$, $u\in U_\gotn$}
               \end{equation}
               where $z^{-\uw}=z^{-w_1}{\bar z}^{-w_\si}$ (the restriction of $\Phi$ to $\C$ is the infinity type of $\Phi$).
               Given $\phi\in H_\uw(\gotn;E)$ there exists a unique $\phi^{(\infty)}\in\arH_\uw(\gotn)$ such that
               $\phi^{(\infty)}(t)=h(I)$ whenever $t_v=1$ for all $v\mid\infty\gotn$ and $I=[t]$.
               The mapping $\phi\mapsto\phi^{(\infty)}$ extends by linearity to an isomorphism
               $H_\uw(\gotn;\C)=H_\uw(\gotn;\overline\Q)\otimes\C\simeq\arH_\uw(\gotn)$.
  \item \textit{$p$-adic avatars.}
           Let $F\supset E$ be a $p$-adic local subfield of $\C_p$
           and $\gotF(\gotn^\ast;F)$ the space of $F$-valued 
           continuous functions on $\gotC_{\gotn^\ast}=\gotC_\gotn=\varprojlim_{j\geq0}\CC_{\gotn p^j}$. 
           Then,  every $\phi\in H_\uw(\gotn^\ast;E)$ defines a unique $\phi^{(p)}\in\gotF(\gotn^\ast;F)$ which
           agrees with $\phi$ on the dense subset $\calI_{\gotn^\ast p}$. The mapping $\phi\mapsto\phi^{(p)}$
           identifies $H_\uw(\gotn^\ast p^r;E)\otimes F$ with the closed subspace 
           $\gotF_{\uw,r}(\gotn^\ast;F)\subset\gotF(\gotn^\ast;F)$ of functions such that
           $\Phi(xt)=x^{-\uw}\Phi(t)$ for all $x\in\calO_K\otimes\Z_p$, $x\equiv 1\bmod \gotn^\ast p^r$.
 \end{itemize}
Note that $\xi\in H_\uw(\gotn;\overline{\Q})$ is a character if and only if its avatars $\xi^{(\infty)}$ and $\xi^{(p)}$
are characters. Let
$$
H(\gotn;E)=\bigoplus_{\uw\in\Z[\text{Gal}(K/\Q)]}H_\uw(\gotn;E).
$$
If $\gotm\mid\gotn$ then $H(\gotm;E)\subset H(\gotn;E)$ and by the linear independence 
of characters the construction of the $p$-adic avatars embeds $H(\gotn\gotq;E)\otimes F$ 
in $\gotF(\gotn^\ast;F)$ when $\gotq$ is supported on primes dividing $p$. The space $H(\gotn;E)$ supports 
an $E$-valued pairing defined as follows. If $\phi_1\in H_{\uw_1}(\gotn;E)$ and 
$\phi_2\in H_{\uw_2}(\gotn;E)$ let
\begin{equation}
\label{eq:pairing}
\langle\phi_1,\phi_2\rangle=
\frac1{h_\gotn}\sum_{s\in\CC_\gotn}\phi_1(I_s)\phi_2(I_s),\qquad
\text{if $\uw_1+\uw_2=\underline{0}$}
\end{equation}
and $\langle\phi_1,\phi_2\rangle=0$ otherwise, where $\{I_s\}$ is a full set of representatives of $\CC_\gotn$.
Then, extend to the whole of $H(\gotn;E)$ by linearity. Note that the pairing is compatible with the inclusions 
$H(\gotn;E)\subset H(\gotn^\prime;E)$ for $\gotn\mid\gotn^\prime$ and that the characters satisfy the
 orthogonality relations $\langle\xi_1,\xi_2\rangle=1$ if $\xi_1\xi_2=1$ and 
 $\langle\xi_1,\xi_2\rangle=0$ otherwise. The pairing can be computed in terms of the avatars.
 \begin{itemize}
       \item \textit{Archimedean pairing.} For all $\Phi_1\in\arH_{\uw_1}(\gotn)$ and  
                $\Phi_2\in\arH_{\uw_2}(\gotn)$ let
                \begin{equation}
                \label{eq:pairinginf}
                \langle\Phi_1,\Phi_2\rangle_\infty=
                \frac1{h_\gotn}\sum_{s\in\CC_\gotn}\Phi_1(t_s)\Phi_2(t_s),\qquad
                \text{if $\uw_1+\uw_2=\underline{0}$}
                \end{equation}
                and $\langle\Phi_1,\Phi_2\rangle_\infty=0$ otherwise, where $\{t_s\}$ is a full set of 
                representatives of $\CC_\gotn$ in $K_\A^\times$. After extending to 
                $\arH(\gotn)=\bigoplus_\uw\arH_\uw(\gotn)$ by linearity, it is clear that
                $\langle\phi_1,\phi_2\rangle=\langle\phi^\infty_1,\phi^\infty_2\rangle_\infty$ for all $\phi_1$,
                $\phi_2\in H(\gotn,E)$ and their corresponding archimedean avatars since it is always 
                possible to take 
                representatives such that $t_{s,v}=1$ for all $v\mid\infty\gotn$, thus reducing \eqref{eq:pairinginf}
                to \eqref{eq:pairing}.
       \item \textit{$p$-adic pairing.} Let 
               $\gotF_r(\gotn^\ast;F)=\widehat{\bigoplus}_{\uw}\gotF_{\uw,r}(\gotn^\ast;F)$. Since 
               $\gotF_{\uw,r}(\gotn^\ast;F)$ is closed, the projection onto the $\uw$-th summand extends 
               to a continuous projection $\pi_{\uw,r}:\gotF_r(\gotn^\ast;F)\map\gotF_{\uw,r}(\gotn^\ast;F)$.
               Define $\langle\cdot,\cdot\rangle_p$ as
               $$
               \xymatrix{
              \gotF_r(\gotn^\ast;F)\times\gotF_r(\gotn^\ast;F)\ar[r]^-{m}\ar@/_2pc/[rrr]^{\langle\cdot,\cdot\rangle_p} & 
              \gotF_r(\gotn^\ast;F)\ar[r]^-{\pi_{\underline{0},r}} &
              \gotF_{\underline{0},r}(\gotn^\ast;F)\ar[r]^-{\mu_H} & F
              },
              $$
              where $m$ is multiplication and $\mu_H$ the Haar distribution 
              (bounded on $\gotF_{\underline{0},r}(\gotn;F)$).
              Then $\langle\phi_1,\phi_2\rangle=\langle\phi^{(p)}_1,\phi^{(p)}_2\rangle_p$ for all $\phi_1$,
              $\phi_2\in H(\gotn,E)$  because $\CC_n$ admits
              representantives in the dense subset 
              $\calI_{\gotn p}$, where $\phi_1=\phi^{(p)}_1$ and $\phi_2=\phi^{(p)}_2$ so that
              the computation of $\langle\phi_1,\phi_2\rangle_p$ reduces to \eqref{eq:pairing}.
 \end{itemize}
 
For a continuous character $\chi:{\widehat\calO}_{K,c}^\times\map\C^\times$,
denote $\arH_\uw(\chi)$ the linear space of functions $\Phi:K_\A^\times\map\C$ such that
\begin{equation}
\label{eq:grosswch}
\Phi(\la tzu)=z^{-\uw}\Phi(t)\chi(u)\qquad
\text{for all $\la\in K^\times$, $z\in\C^\times$, $u\in{\widehat\calO}_{K,c}^\times$.}
\end{equation}
Let $\gotn_\chi$ be the smallest ideal such that $\arH_\uw(\chi)\subset\arH_\uw(\gotn_\chi)$.
Note that $\gotn_\chi$
is at least divisible by the primes $\gotp$ not in the support of  $c$ and such that
$\chi_{|(\calO_{K,c})_\gotp^\times}\not\equiv1$ and may be divisible by primes $\gotp$ in the support of $c$
even if $\chi_{|(\calO_{K,c})_\gotp^\times}\equiv1$.
For $\Phi_1\in\arH_{\uw_1}(\chi_1)$ and $\Phi_2\in\arH_{\uw_2}(\chi_2)$
with $\uw_1+\uw_2=0$ and $\chi_1\chi_2=1$
the pairing simplifies to
$$
\langle\Phi_1,\Phi_2\rangle_\infty=
\frac1{h_c^\sharp}\sum_{s\in\CC_c^\sharp}\Phi_1(t_s)\Phi_2(t_s),
$$
where $\{t_s\}$ is now a full set of  representatives of 
$\CC_c^\sharp$ in $K_\A^\times$, which
we can assume to be $p$-normalized, i.e. 
$t_{s,\infty}=1$, $t_{s,\ell}\in\calO_K\otimes\Z_\ell$ for all primes $\ell$
and furthermore $t_{s,v}$ a $v$-unit for all places $v\mid pc$, for all $s\in\CC_c^\sharp$.

\bigskip
%We now collect a few facts about measures that shall be used later.
Let $\mu_1$, $\mu_2$ be measures on $\Z_p$ with values in the $p$-adic Banach $F$-space $\gotB$.
For a continuous function $h:\Z_p^2\map\Z_p$ let $\mu_h=h_*(\mu_1\otimes\mu_2)$
the $\gotB{\widehat\otimes}_F\gotB$-valued measure on $\Z_p$ such that
$\int_{\Z_p}\varphi(t)\,d\mu_h(t)=\iint_{\Z_p^2}\varphi(h(t_1,t_2))\,d(\mu_1\otimes\mu_2)(t_1,t_2)$
for all $\varphi\in\calC^0(\Z_p,F)$.
When $h=m$ is the multiplication
%$(t_1,t_2)=t_1t_2$, 
the moments of the measure $\mu_m$ are simply
$$
m_n(\mu_m)=\int_{\Z_p}t^n\,d\mu_m=\iint_{\Z_p^2}(t_1t_2)^n\,d(\mu_1\otimes\mu_2)(t_1,t_2)
=m_n(\mu_1)\otimes m_n(\mu_2)
$$
for all $n\geq0$. In particular, if $\gotB=F$ and $\mu_2=\de_z$ is the Dirac measure at $z\in\Z_p$ then 
$m_n(\mu_m)=m_n(\mu_1)z^n$.
Also, $t_1t_2\in\Z_p^\times$ if and only if $(t_1,t_2)\in(\Z_p^\times)^2$ so that
$$
\mu_m^\times=m_*(\mu_1^\times\otimes\mu_2^\times).
$$

%Let $\mu_1$, $\mu_2$ be two measures on $\Z_p$ with values in the $p$-adic field $F$. Given a 
%continuous function $\al:\Z_p^2\map\Z_p$ let $\mu_\al=\al_*(\mu_1\times\mu_2)$ the 
%$F$-valued measure on $\Z_p$ such that
%$\int_{\Z_p}\varphi(t)\,d\mu_\al(t)=\iint_{\Z_p^2}\varphi(\al(t_1,t_2))\,d(\mu_1\times\mu_2)(t_1,t_2)$
%for all $\varphi\in\calC^0(\Z_p,F)$. When $\al(t_1,t_2)=t_1t_2$, the measure $\mu_\al$ has the property 
%that
%$$
%m_n(\mu_\al)=\int_{\Z_p}t^r\,d\mu_\al=\iint_{\Z_p^2}(t_1t_2)^r\,d(\mu_1\times\mu_2)(t_1,t_2)
%=m_n(\mu_1)m_n(\mu_2)
%$$
%for all $n\geq0$. In particular, if $\mu_2=\de_z$ is the Dirac measure at $z\in\Z_p$ then 
%$m_n(\mu_\al)=m_n(\mu_1)z^n$.
%Also, $t_1t_2\in\Z_p^\times$ if and only if $(t_1,t_2)\in(\Z_p^\times)^2$ so that
%$$
%\mu_\al^\times=\al_*(\mu_1^\times\times\mu_2^\times).
%$$

\begin{pro}\label{pairedmeas}
Let $\mu_1$ and $\mu_2$ be $\gotF(\gotn^\ast;F)$-valued measures on $\Z_p$
such that $m_n(\mu_i)=\int_{\Z_p}t^n\,d\mu_i(t)=\la_{n,i}\phi^{(p)}_{n,i}$ for all $n\geq0$ and $i\in\{1,2\}$,  
where $\la_{n,i}\in F$ and $\phi^{(p)}_{n,i}$ is the $p$-adic avatar of some $\phi_{n,i}\in H_{\uw_{n,i}}(\gotn;E)$ 
for a subfield $E\subset\C$ 
admitting an embedding into $F$ and such that $\uw_{n,1}+\uw_{n,2}=\underline 0$. Then there exists an 
$F$-valued measure $\mu=\mu_{\langle\mu_1,\mu_2\rangle}$ on $\Z_p$ such that
$$
m_n(\mu)=\int_{\Z_p}t^n\,d\mu(t)
=\la_{n,1}\la_{n,2}\langle\phi_{n,1},\phi_{n,2}\rangle.
$$
Moreover, $\mu^\times=\mu_{\langle\mu_1^\times,\mu_2^\times\rangle}$.
\end{pro}

\begin{proof}
Fix $r\geq0$ such that $\gotn\mid\gotn^\ast p^r$: it is clear that 
$\int_{\Z_p}\varphi\,d\mu_i\in\gotF_r(\gotn^\ast;F)$ for all $\varphi\in\calC^0(\Z_p,F)$.
Since  the Haar distribution is bounded on $\gotF_0(\gotn^\ast;F)$,
the pairing $\langle\cdot,\cdot\rangle_p$ is bounded on $\gotF_r(\gotn^\ast;F)\times\gotF_r(\gotn^\ast;F)$.
If $F_p$ denotes the bounded linear functional on $\gotF_r(\gotn^\ast;F)\widehat\otimes\gotF_r(\gotn^\ast;F)$
corresponding to  $\langle\cdot,\cdot\rangle_p$
the setting
$$
\int_{\Z_p}\varphi(t)\,d\mu(t)=
%\mu_H\circ\pi_{\underline{0},r}
F_p\left(\iint_{\Z_p^2}\varphi(t_1t_2)\,d(\hat\mu_1\times\hat\mu_2)(t_1,t_2)\right)
$$
defines a measure on $\Z_p$ which has the stated properties because of the preceding discussion.
\end{proof}

\begin{rem}\rm
Composing either measure $\mu_i$ with the evaluation map at any $I\in\calI_{\gotn}$ defines
$F$-valued measures $\mu_i^I$ such that $m_n(\mu_i^I)=\la_{n,i}\phi^{(p)}_{n.i}(I)$. Taking a system of representatives of
$\CC_\gotn$ in $\calI_{\gotn^\ast p}$ yields
\begin{multline*}
\langle m_r(\mu_1),m_r(\mu_2)\rangle_p=\\
\frac{\la_{r,1}\la_{r,2}}{h_\gotn}\sum_{s\in\CC_\gotn}\phi_{r,1}(I_s)\phi_{r,2}(I_s)=
\frac1{h_\gotn}\sum_{s\in\CC_\gotn}m_r(m_*(\mu_1^{I_s}\times\mu_2^{I_s})).
\end{multline*}
By additivity of measures, this provides an alternate method of constructing 
$\mu_{\langle\mu_1,\mu_2\rangle}$.
\end{rem}

\begin{rem}\label{rem:twists}\rm
Given a character $\psi:\CC_\gotn\map\C^\times$ the formula
$$
\langle\phi_1,\phi_2\rangle^\psi=\langle\phi_1,\psi\phi_2\rangle
$$
defines a twisted version of the pairing \eqref{eq:pairing} 
which has likewise archimedean and $p$-adic counterparts. 
The same argument of the proof of proposition \ref{pairedmeas} yields a measure $\mu^\psi$ on $\Z_p$
such that $\int_{\Z_p}t^n\,d\mu(t)=\la_{n,1}\la_{n,2}\langle\phi_{n,1},\phi_{n,2}\rangle^\psi$.
The measures $\mu^\psi$ can be read as a measure $\mu$ on $\CC_\gotn\times\Z_p$ such that
$$
\int_{\CC_\gotn\times\Z_p}\psi(s)t^n\,d\mu(s,t)=\la_{n,1}\la_{n,2}\langle\phi_{n,1},\psi\phi_{n,2}\rangle
$$

\end{rem}
 
We shall now define the measure which will be used for the interpolation process in the next section.
Let $\rho\in\Hom^\sharp(K,D)$ with corresponding point $x_0=[\tau]\in\CM(\De,N;\calO_{K,c})$.
Its adelic extension $\rho_\A$ gives rise to an inclusion
$$
 \bar\rho_\A:\CC_c^\sharp
 %=K_\A^\times/K^\times\C^\times\widehat\calO_{K,c}^\times
 \mmap
 D^\times\bs D_\A^\times/\rho(\C^\times)\widehat\calR_N^\times
 \hookrightarrow X^\De_0(N)
 $$
 (the last map is actually an equality when $\De>1$).
 Under the decomposition
 \begin{equation}
\label{eq:decompDstar}
D_\A^\times=D_\Q^\times\GL_2(\R)^+\widehat\calR_N^\times
\end{equation}
the idele $d=d_\Q g_\infty\hat r$ corresponds to the point represented by $g_\infty\tau$.
The canonical map of class field theory
$K_\A^\times\mmap\Gal(K^{\rm ab}/K)$, where $t\mapsto[t,K]$, yields a canonical identification 
$\CC_c^\sharp\simeq\Gal(H_c/K)$
where $H_c$ is the ring class field of conductor $c$. 
The points in ${\rm Im}(\bar\rho_\A)$ are defined over $H_c$ and by Shimura's 
reciprocity law the two Galois actions are compatible. In particular, if $s=[t,K]_{\left|H_c\right.}$
then $\bar\rho_\A(\bar t)=x_0^s$ and $A_{x_0^s}^s\simeq A_{x_0}$.
Also, for a choice of $p$-normalized representatives $\{t_s\}$ the theory of complex multiplications
yields isogenies of complex tori of prime to $p$  degree
$$
\xymatrix{A_{g_s\tau}=\C^\ep/\La_{g_s\tau}\ar[r]^-{j(g_s,\tau)}_-\sim &
A_{x_0^s}(\C)=\C^\ep/t_s\La_{\tau}\ar[r]^-{\pi_{t_s}} & A_{x_0}(\C)=\C^\ep/\La_{\tau}}
$$
($\ep=1$ or $2$) where $\pi_{t_s}$ is a quotient map.
Let $w\mid p$ a place in a large enough field $L$ such that $(x^s,w,e)$ is a 
$p$-ordinary split arithmetic object for $(\De,N,K)$ for all $s\in\CC_c^\sharp$ and for all $x^s$
mapping to $x_0^s$ under the natural map $X_1^\De(N)\map X_0^\De(N)$,
and also such that the isogenies $\pi_{t_s}$ are defined over $L$.
Fix an invariant $1$-form $\om_o$ on $A_x$ as in \eqref{eq:nerondiff} and let $\om_{t_s}=\pi_{t_s}^*(\om_o)$.
Then $\om_{t_s}$ satisfies \eqref{eq:nerondiff} for $A_{x_s}$ so that it may be assumed that
\begin{equation}\label{eq:Omrelation}
\Om_\infty(g_s\tau)=j(g_s,\tau)\Om_\infty(\tau)\quad\text{and}\quad
\Om_p(x^s)=\Om_p(x).
\end{equation}
If $t_s^\prime=t_s\la zu$ is another $p$-normalized representative of  $s\in\CC_c^\sharp$, then
comparing decompositions \eqref{eq:decompDstar} yields
$\om_{t^\prime_s}\sim_{\calO_{K,c}^\times}z\om_{t_s}$. Thus, the construction of $\om_t$ can be extended,
up to $\calO_{K,c}^\times$-association, to all of $K_\A^\times$ by setting
$\om_{t\la zu}\sim_{\calO_{K,c}^\times}z\om_t$
for all $\la\in K^\times$, $z\in\C^\times$, $u\in\calO_{K,c}^\times$ where $t$ is $p$-normalized.
%Suppose that $x_0=x_{0,\nu}$ corresponds, in the above sense, to the $\nu$-th point 
%$x_\nu\in X_1^\De(Mp^\nu)$ in a compatible family $\calT_j$ of objects  as in section 2.2.
%Let $c_\nu$ be the conductor of $x_0$. Since $c_{\nu+1}=p^\ep c_\nu$, 
%with $\ep\in\{0,1\}$ for all $\nu\geq0$ the prime-to-$p$ part $c^\ast$ of $c_\nu$
%is independent of $\nu$. To simplify notation, the subscript $\nu$ will be dropped whenever possible.

Fix a point $x\in X_1^\De(N)$ such that $x\mapsto x_0$ under the natural map and assume that $x$ belongs to
a $p$-ordinary split arithmetic object (this depends only on $x_0$).
Let  $f\in M_{0,2\ka}^\De(N,\vep;\calO_{(v)})$. The nebentypus $\vep$ gives rise to 
a character $\vep_\A$ of $\calR_{N}^\times$  and thus to a character 
$\vep_\rho$ of $\widehat\calO_{K,c}^\times$ via the embedding $\rho$.
For all $r\geq0$ such that $(\calO_{K,c}^\times)^{2(\ka+r)}=1$ define a complex valued function
$\vth^{(\infty)}(r,f,x)$ on $K_\A^\times$ as
$$
\vth^{(\infty)}(r,f,x)(t)=\frac{\de_{2\ka}^{(r)}f(g_\infty\tau)\vep_\A(\hat r)}{\perr(\om_t,g_\infty\tau)^{2(\ka+r)}},
$$
where $\rho_\A(t)=d_0g_\infty\hat r$ under \eqref{eq:decompDstar}.

\begin{pro}
$\vth^{(\infty)}(r,f,x)\in\arH_{(2(\ka+r),0)}(\vep_\rho)$.
Moreover, $\vth^{(\infty)}(r,f,x)$ is the archimedean avatar of 
$\vth(r,f,x)\in H_{(2(\ka+r),0)}(\gotn_{\vep_\rho};L)$ and its $p$-adic avatar
$\vth^{(p)}(r,f,x)$ is in $\gotS(\gotn_{\vep_\rho}^\ast;L_w)$, the unitary ball in
$\gotF(\gotn_{\vep_\rho}^\ast;L_w)$ under the sup norm.
\end{pro}

\begin{proof}
The function $\vth^{(\infty)}(r,f,x)$ satisfies \eqref{eq:grosswch} for $\uw=(2(\ka+r),0)$ and 
$\chi=\vep_\rho$ simply from the definition of $\om_t$ and because if 
$\rho_\A(t)=d_0g_\infty\hat r$ under \eqref{eq:decompDstar}, 
then $\rho_\A(t\la zu)=(d_0\la)(g_\infty z)(\hat ru)$ and $\rho_\infty(z)$ fixes $\tau$.

If $I=[t]\in\calI_{\gotn_{\vep_\rho}}$ with $t_v=1$ for all $v\mid\infty\gotn_{\vep_\rho}$, write $t=t_s\la zu$ for some 
$\la\in K^\times$, $z\in\C^\times$, $u\in\widehat\calO_{K,c}^\times$ and $t_s$ $p$-normalized. Then 
$\la z=1$, so that
\begin{equation}
\label{eq:valueofvth}
\vth(r,f,x)(I)=\vth^{(\infty)}(r,f,\rho)(t_s\la zu)=\la^{2(\ka+r)}\th_w^{(r)}(f,x^s)
\end{equation}
for some $X_1^\De(N)\ni x^s\mapsto x_0^s$ and  theorem \ref{integcoeff} applies.
Finally, if the ideal $I$ is also prime to $p$ the right hand side in \eqref{eq:valueofvth} 
is in $\calO_{(w)}$, again by 
theorem \ref{integcoeff}, because $\la$ is a $p$-unit.
\end{proof}

With a slight abuse of notation, denote $L_w^\prime$ the completion of $L_w(\Om_p)$
with ring of integers $\calO^\prime_w$.
\begin{pro}\label{pr:meastheta}
There is a unique $\gotS(\gotn_{\vep_\rho}^\ast;L_w^\prime)$-valued measure 
$\mu_{f,\rho}$ on $\Z_p$ such that
$$
m_r(\mu_{f,\rho})=\int_{\Z_p}t^r\,d\mu_{f,\rho}(t)=\Om_p^{2(\ka+r)}\vth^{(p)}(r,f,x),\qquad
\text{for all $r\geq0$.}
$$
\end{pro}
\begin{proof}
Formula \eqref{eq:valueofvth} together with the fact that $\la\in\Z_p^\times$ because 
$p$ splits in $K$ says that on the dense subset $\calI_{p\gotn_{\vep_\rho}}$
the value $\Om_p^{2(\ka+r)}\vth(r,f,x)(I)=\la^{2\ka}m_r(\mu_{f,x^s})$ 
is the $r$-th moment of a $L^\prime_w$-valued measure $\mu_I$
on $\Z_p$. It is clear by continuity and the existence of the $p$-adic avatars $\vth^{(p)}(r,f,x)$ that for any 
continuous $\calO_w^\prime$-valued function $\varphi$ on $\Z_p$ the function
$I\mapsto\int_{\Z_p}\varphi(t)\,d\mu_I(t)$
extends uniquely to $\gotC_{\gotn_{\vep_\rho}^\ast}$ to a continuous $\calO_w^\prime$-valued function.
\end{proof}

\begin{cor}\label{measthstar}
$m_r(\mu_{f,\rho}^\times)=\Om_p^{2(\ka+r)}\vth^{(p)}(r,(1-VU)f,x)$.
\end{cor}

\begin{proof}
This is a direct application of theorem \ref{th:mustar}.
\end{proof}

\begin{rem}\label{rem:onchar}\rm
In the next section the functions $\vth(r,f,x)$, or any of their avatars, will be paired against characters.
Proposition \ref{pairedmeas} suggests to consider   $\gotS(\gotn_{\vep_\rho}^\ast;F_w)$-valued measures
with moments characters of weight $(-2(\ka+r),0)$. A basic way to obtain such measures is as 
follows. Let $\chi_0\in H_{(-2\ka,0)}(\gotn;E_0)$ and $\chi\in H_{(2,0)}(\gotn;E)$ be characters such that 
$p$ splits completely in $E$ and let $\chi_r=\chi_0\chi^r$ for all $r\geq0$. Let $F$ be a $p$-adic field 
containing both $E_0$ and $E$. Then there exists a $\gotS(\gotn;F)$-valued meaure $\mu_{\chi_0,\chi}$ on 
$\Z_p$ such that
$$
m_r(\mu_{\chi_0,\chi})=\int_{\Z_p}t^r\,d\mu_{\chi_0,\chi}=\chi_r^{(p)},\qquad
\text{for all $r\geq0$}.
$$
Indeed, evaluation at every $s\in\gotC_\gotn$ defines a scalar valued $p$-adic distribution $\mu_{\chi_0,\chi}(s)$
whose associated power series is
$$
\Phi_{\mu_{\chi_0,\chi}(s)}(Z)=\sum_{r=0}^\infty\frac{\chi_r^{(p)}(s)}{r!}Z^r=\chi_0^{(p)}(s)(1+T)^{\chi^{(p)}(s)}
\qquad (T=e^Z-1),
$$
which has bounded coefficients in the variable $T$ 
since our assumption on $p$ implies that $\chi^{(p)}$ takes values in $\Z_p^\times$
(in fact, $\mu_{\chi_0,\chi}(s)$ is simply a scalar multiple of the Dirac measure at $\chi(s)^{(p)}$).
As in the previous proof, this entails that $\mu_{\chi_0,\chi}$ is a measure which in fact is supported on $\Z_p^\times$, i.e
$\mu_{\chi_0,\chi}^\times=\mu_{\chi_0,\chi}$.
%it follows from the assumption on $p$ that $\hat\chi$ takes values in $\Z_p^\times$. Thus for all
%$s\in\gotC_\gotn$ the series
%$$
%\sum_{r=0}^\infty\frac{\hat\chi_r(s)}{r!}Z^r=\hat\chi_0(s)(1+T)^{\hat\chi(s)},
%\qquad
%T=e^Z-1=Z+\frac1{2!}Z^2+\frac1{3!}Z^3+\cdots,
%$$
%has bounded coefficients in the variable $T$. As in the previous proof, this entails that $\hat\mu$
%is a measure.
\end{rem}

% ********  SECTION THREE  **************

\section{Interpolation}

\subsection{Some periods computations}
For any $f^\De\in M_{0,k,\infty}^\De(N,\vep)$ let 
$F_{f^\De}\in L^2(D_\Q^\times\backslash D_\A^\times,\vep_\A)$
the usual adelic lift of $f^\De$, namely
$F_{f^\De}(d)=f^\De(g_\infty\cdot i)j(g_\infty,i)^{-k}\det(g_\infty)^{k/2}\vep_\A(\hat r)$
if $d=d_\Q g_\infty\hat r$ under \eqref{eq:decompDstar}.
The Lie algebra $\gotg=\gotg\gotl_2\simeq\text{Lie}(D_\infty^\times)$ acts on the $\C$-valued 
$\calC^\infty$ functions on $D_\A^\times$ by
$A\ast\varphi(d)=\left.\frac d{dt}\varphi(de^{tA})\right|_{t=0}$ and the action extends to the complexified 
universal enveloping algebra $\gotA(\gotg)_\C$. Let
$I=\smallmat1001$, $H=\smallmat0{-i}i0$, $X^\pm=\frac12\smallmat1{\pm i}{\pm i}{-1}$
be the standard eigenbasis of $\gotg_{\C}$ for the adjoint action of the maximal compact subgroup
$\SO(2)=\left\{\hbox{$r_\th=\smallmat{\cos\th}{-\sin\th}{\sin\th}{\cos\th}$ such that $\th\in\R$}\right\}$.
Since ${\rm Ad}(r_\th)X^\pm=e^{\mp 2i\th}X^\pm$, the function $X^\pm\ast\phi_F$ is a lift from 
$M_{0,k\pm2,\infty}^\De(N,\vep)$. A standard computation (e.g. \cite[\S\S2.1--2]{Bump96})
links the Lie action to the archimedean Maass operators of \S\ref{diffop}, namely
$X^+\ast F_{f^\De}=-4\pi\phi_{\de_{2\ka}{f^\De}}$.
For $r\geq0$ let
\begin{equation}\label{def:phir}
F_{f^\De}^{(r)}=\left(-\frac{1}{4\pi}X^{+}\right)^{r}\ast F_{f^\De}=F_{\de_{2\ka}^{(r)}f^\De}.
\end{equation}
\begin{dfn}\label{th:Jintegral}
	\rm 
	Let $f^\De\in M_{0,2\ka}^\De(N,\vep)$, 
	$\rho\in\Hom^\sharp(K,D)$ of conductor $c$ with fixed point
	$\tau\in\CM_{\De,K}$ and $\xi$ a character of $K_\A^\times$ trivial on 
	$K^\times\R^\times$. For each $r\geq0$ let
	$$
	%\label{eq:intJ}
	J_r(f,\xi,\tau)=\int_{K_\A^\times/K^\times\R^\times}F_{f^\De}^{(r)}(\rho(t)d^\tau_\infty)\xi(t)\,dt
	$$
	where $dt$ is the Haar measure on $K_\A^\times$ whose archimedean component 
	is normalized so that $\text{vol}(\C^\times/\R^\times)=\pi$ and such that the local groups of units 
	have volume 1. 
\end{dfn}

%\begin{rem}\label{rem:Jiszero}\rm
%     Let $U$ be the largest subgroup of finite ideles in $K_A^\times$ such that $F_f^{(r)}(\rho(t))$
%     is right $\chi$-invariant for some character $\chi$ of $U$. Then a simple integration by substitution shows that
%     $J_r(f,\xi,\rho)=0$ when $\chi\xi_{\left|U\right.}\neq1$.
%     We shall implicitly assume that  $\hat\calO_{K,c}^\times=U$ (this is typically not the case for oldforms).
%\end{rem}

Note that $J_r(f,\xi,\tau)=0$ when  $\chi\xi_{\left|U\right.}\neq1$ where $U$ is the largest subgroup of finite ideles
in $K_A^\times$ such that $F_{f^\De}^{(r)}(\rho(t))$ is right $\chi$-invariant for some character $\chi$ of $U$.
 We shall implicitly assume that  $\hat\calO_{K,c}^\times=U$
 %(this is typically not the case for oldforms).
The following proposition relates $J_r(f,\xi,\tau)$ to the pairing \eqref{eq:pairinginf}.

\begin{pro}\label{pr:Jasscpr}
In the situation of definition \ref{th:Jintegral} assume $\xi\in\arH_\uw(c,\vep_\rho^{-1})$
with $\uw=(w,-w)$ and that $(\calO_{K,c}^\times)^{2(\ka+r)}=1$. Then
$$
J_r(f,\xi,\tau)=2\pi\frac{\mathrm{vol}(\hat\calO_{K,c}^\times)h_c^\sharp}{\vass{\calO_{K,c}^\times}}
\frac1{(\mathrm{Im}(\tau)\Om_\infty(\tau)^2)^w}
\langle\vth^{(\infty)}(r,f,x),\xi\vvass{\text{N}_{K/\Q}}^{-w}\rangle_\infty
$$
% if $w=-\ka-r$ and $J_r(f,\xi,\rho)=0$ otherwise.
\end{pro}

\begin{proof}
% When $w\neq-\ka-r$ the formula follows immediately from remark \ref{rem:Jiszero} and the fact that the scalar product 
% on the right hand side vanishes due of  incompatibilities of weights. So let us assume that $w=-\ka-r$.
The integrand is right $\hat\calO_{K,c}^\times$-invariant. We can write
$K_A^\times/K^\times\R^\times\hat\calO_{K,c}^\times=
\bigcup_{s\in\CC_c^\sharp}t_s\C^\times/\R^\times\calO_{K,c}^\times$ 
as a disjoint union for some set of representatives 
%$\{t_s\}$ of $\CC_c^\sharp$ 
$$
J_r(f,\xi,\tau)=\mathrm{vol}(\hat\calO_{K,c}^\times)\sum_{s\in\CC_c^\sharp}\xi(t_s)
\int_{\C^\times/\R^\times\calO_{K,c}^\times}
F_{f^\De}^{(r)}(\rho(t_sz)d^\tau_\infty)\xi_\infty(z)\,d^\times z.
$$
To compute each summand write $z=re^{i\th}$ and 
$\rho_\infty(z)=rd_\infty^\tau r_\th(d_\infty^\tau)^{-1}$ so that
$$
\int_{\C^\times/R^\times\calO_{K,c}^\times}F_{f^\De}^{(r)}(\rho(t_sz)d^\tau_\infty)\xi_\infty(z)\,d^\times z=
\frac{F_{f^\De}^{(r)}(\rho(t_s)d^\tau_\infty)}{\vass{\calO_{K,c}^\times}}
\int_0^{2\pi}e^{-2i(\ka+r+w)\th}\,d\th.
$$ 
Thus $J_r(f,\xi,\tau)=0$ if $w\neq-\ka-r$ proving the formula in this case since also the scalar product 
vanishes due to the incompatibility of weights.
To finish the proof, write $t_s=d_sg_su_s$ under \eqref{eq:decompDstar}.
Then, under the hypotheses that $t_s$ is normalized and  $(\calO_{K,c}^\times)^{2{\ka+r}}=1$,
the first identity in \eqref{eq:Omrelation} implies that
$$
\xi(t_s)F_{f^\De}^{(r)}(t_sd^\tau_\infty)=
\mathrm{Im}(\tau)^{\ka+r}\Om_\infty(\tau)^{2(\ka+r)}\vth^{(\infty)}(r,f^\De,x)\xi(t_s)
\vvass{\text{N}_{K/\Q}(t_s)}^{\ka+r}
$$
since $\det(g_s)=\vvass{\text{N}_{K/\Q}(t_s)}$.
\end{proof}

We shall henceforth assume that $\xi$ satisfies the hypothesis in proposition \ref{pr:Jasscpr}. Denote $\gotp^s$ 
the conductor of the local component $\xi_\gotp$. For the next proposition assume also that the CM point $\tau$ 
represents a point $x\in\XarD(Mp^\nu)(\calO_{v_\nu})$ in a $p$-ordinary split arithmetic test object, with $(M,p)=1$.

\begin{pro}
Suppose that $(p,c)=1$.
Then, there exists a uniformizer $\vpi_\gotp$ of $\gotp\calO_\gotp$, well-defined up to 
$(1+\gotp^s)$-association such that
$$
J_r(Vf^\De,\xi,\tau)=p^{-2\ka}\xi_\gotp(\vpi_\gotp)J_r(f^\De,\xi,\tau).
$$
\end{pro}

\begin{proof}
Over $\C$, where modular forms can be regarded as homogeneous functions of lattices and level structures,
$Vf^\De(P,\La)=p^{-2\ka}f^\De(P^\prime,\La^\prime)$ where $P$ is a point of period 
$Mp^{\nu+1}$ in $A=\C^\vep/\La$, $\La^\prime$ is the lattice ($\calR$-module when $\De>1$) generated by 
$\La$ and $\bar P=Mp^\nu P$, and $P^\prime$ the image of $P$ in 
$A^\prime=\C^\vep/\La^\prime$ under the natural quotient map (clearly a point of period $Mp^\nu$ there).

In our situation of arithmetic structures, 
the group ($\calR$-module) generated by $\bar P$ in $A_\tau=\C^\vep/\La_\tau$
is the base change to $\C$ of $\eta_p(\mu_p^\vep)$. Thus, if $\si_\gotp=(\gotp,H_c/K)$ is the 
Frobenius at $\gotp$, 
$\C^\vep/\La_\tau^\prime=A_x^{\si_\gotp}(\C)$ and the natural quotient map
induced by the inclusion $\La_\tau\subset\La_\tau^\prime$ is a lifting to $\C$ of the Frobenius map 
between the varieties reduced $\bmod\gotp$.

Since $(p,c)=1$, we can write $\si_\gotp=[\widetilde\varpi_\gotp,K]_{\left|H_c\right.}$ where 
$\widetilde\varpi_\gotp\in K_\A^\times$ has component a uniformizer $\varpi_\gotp$ at $\gotp$ and $1$ elsewhere. 
Again by Shimura reciprocity we may take $\La_\tau^\prime=\widetilde\varpi_\gotp^{-1}\La_\tau$ so that $\bar P$ 
is a generator  (as a group or as $\calR$-module) of 
$\widetilde\varpi_\gotp^{-1}\La_\tau/\La_\tau$ and we also know that the action of 
$[\widetilde\varpi_\gotp,K]$ on division points is given by the \lq\lq multiplication by $\widetilde\varpi_\gotp^{-1}$\rq\rq map
\begin{equation}
\label{eq:galdiv}
\La_\tau\otimes\Q/\La_\tau\stackrel{\widetilde\varpi_\gotp^{-1}}\mmap
\La_\tau\otimes\Q/\widetilde\varpi_\gotp^{-1}\Lambda_\tau\hookrightarrow A_x^{\si_p}(\C),
\end{equation}
which is actually defined locally.  Forgetting one $p$-level in the structure simply amounts to moving from 
$(P,\La)$ to $(pP,\Lambda)$ and so choose $\varpi_\gotp$ so that $pP\mapsto P^\prime$ under \eqref{eq:galdiv}. 
This choice of local uniformizer is well-defined only up to a local unit in the stabilizer of $pP$ 
(i.e. the kernel of the $p$-orientations, \S\ref{backMSC}), and for our purposes up to a local unit in the kernel of 
$\tilde\xi_\gotp$ as this coincides with the kernel of the local component of the character induced by the nebentypus.
It follows immediately from the commutativity of the relevant Galois groups that the same uniformizer $\varpi_\gotp$ 
works also replacing $x$ with any other point in in the fiber of ${\rm Im}(\bar\rho_\A)\subset X_0^\De(Mp^\nu)$, i.e. on this
fiber the modular definition of $p^{2\ka}V$corresponds to translation by $\widetilde\varpi_\gotp^{-1}\in K_\A^\times$.
Thus
$$
F_{Vf^\De}^{(r)}(\rho(t)d^\tau_\infty)=p^{-2\ka}F_{f^\De}^{(r)}(\rho(t\widetilde\varpi_\gotp^{-1})d^\tau_\infty)
$$
and the formula now follows readily from an integration by substitution.
\end{proof}

\begin{cor}\label{eulerfact}
Suppose in addition that $f^\De$ is a newform with $T_p(f^\De)=a_pf^\De$. Then
$$
J_r((1-VU)f^\De,\xi,\tau)=
(1-a_p\xi_\gotp(\vpi_\gotp)p^{-2\ka}+\vep(p)\xi_\gotp(\vpi_\gotp)^2p^{-2\ka-1})J_r(f^\De,\xi,\tau).
$$
\end{cor}

\begin{proof}
This follows immediately form the previous proposition recalling that $T_p=U+\vep(p)p^{2\ka-1}V$.
\end{proof}

\subsection{Central $L$-values}\label{sec:Lvalues}
Let $\GO(Q)$ denote the algebraic group of similitudes of the quadratic space $Q$.
%i.e. the group of linear transformations 
%$g\in\GL(Q)$ such that $(gx,gy)_Q=\la_0(g)(x,y)_Q$ for all $x$, $y\in Q$ for some $\la_0(g)\in\C$ 
%where $(\cdot,\cdot)_Q$ is the associated pairing, 
%with Zariski connected component $\GO(Q)^\circ$. 
When $Q=D$ with respect to the reduced norm $\nu$ 
the group $\GO(D)$ is completely described by the short exact sequence
\begin{equation}\label{eq:GODstr}
1\mmap\G_m\mmap D^\times\times D^\times\stackrel\varrho\mmap\GO(D)^\circ\mmap1,\quad
\varrho(x,y)(d)=xd\inv y,
\end{equation}
together with $\GO(D)=\GO(D)^o\ltimes\langle\mathbf{t}\rangle$, where $\mathbf{t}(d)=\bar d$, e.g.
\cite[\S1.1]{Harris93},  \cite[\S7]{HaKu91}. The decomposition $D=K\oplus Ku$ associated to  
$\rho\in\Hom^\sharp(K,D)$ is orthogonal and the norm $\nu$ restricts to $\text{N}_{K/\Q}$
and $-u^2\text{N}_{K/\Q}$ on $K$ and $Ku$ respectively. 
Let $T=\text{G}(\text{O}(K)\times\text{O}(Ku))<\GO(D)$ be the subgroup preserving the splitting.
Left multiplication gives isomorphisms
$\GO(K)^\circ\simeq R_{K/\Q}\G_{m,K}\simeq\GO(Ku)^\circ$ and there is an identification
$$
T^\circ=\left\{
\text{$(k_1,k_2)\in R_{K/\Q}\G_{m,K}^2$ such that $\text{N}_{K/\Q}(k_1\inv k_2)=1$}\right\}.
$$
Also, there is a commutative diagram
\begin{equation}
\label{eq:GOD}
\xymatrix{
K^\times\times K^\times\ar[r]^\al\ar[d]_{\rho\times\rho} & T^\circ\ar[d] \\
D^\times\times D^\times\ar[r]^\varrho & \GO(D)^\circ}
\end{equation}
where $\al(k_1,k_2)=(k_1\inv k_2,k_1\inv{\bar k}_2)$.

Let $\calS(D)_\A=\bigotimes_{\ell\leq\infty}\calS(D)_\ell$ denote the space of adelic 
Schwartz-Bruhat functions whose archimedean component consists of classical 
Schwartz functions on $D_\infty$ which are 
$K^1_\infty\times K^1_\infty$-finite under the action of $\varrho$, where 
$K^1_\infty$ is the maximal compact 
subgroup of $\rho(K\otimes\R)\subset D_\infty^\times\simeq\GL_2(\R)$.
The space $\calS(D)_\A$ supports the representation $\tilde r_\psi$ of the adelic points of the group  
$$
R(D)=\left\{(g,h)\in\GL_2\times\GO(D)\ \mid\ \det(g)=\nu_{0}(h)\right\}
\simeq\SL_2\ltimes\GO(D)
$$
defined to be
$$
\tilde r_{\psi}(g,h)\varphi=r_\psi\left(g\smallmat{1}{}{}{\nu_{0}(h)}^{-1}\right)L(h)\varphi=
L(h)r_\psi\left(\smallmat{1}{}{}{\nu_{0}(h)}^{-1}g\right)\varphi
$$
where $r_\psi$ is the Weil representation of $SL_2(\A)$ and $L$ the unitary representation 
of $\GO(D)(\A)$ given by $L(h)\varphi(x)=\vvass{\nu_o(h)}^{-1}_\A\varphi(\inv{h}x)$ for all $x\in D_\A$,
\cite[\S5]{HaKu92}.
The group $R(D)$  is naturally a subgroup of $\Sp(\PP\otimes D)$, with 
$\PP$ the hyperbolic plane, via $(g,h)x\otimes y=gx\otimes h^{-1}y$.
The groups $(\SL_2,{\rm O}(D))$ form a dual reductive pair in 
$\Sp(\PP\otimes D)$ and the extended Weil representation $\tilde r_{\psi}$ allows to realize the 
theta correspondence between the similitude groups by integrating the theta kernel
$$
\vartheta(g,h;\varphi)=\sum_{d\in D}\tilde r_{\psi}(g,h)\varphi(d)
$$
against automorphic forms. Namely, if  $F$ is a cuspidal automorphic form on $\GL_2$ 
the function defined as
\begin{equation}
     \theta_{\varphi}(F)(h)=\int_{\SL_{2}(\Q)\backslash\SL_{2}(\A)}
     \vartheta(g^{\prime}g,h;\varphi)F(g^{\prime}g)\,dg^{\prime}
     \label{eq:thetalift}
\end{equation}
for any choice of $g\in\GL_2(\A)$ such that $(g,h)\in R(D)_\A$ is an automorphic form on $\GO(D)$. 
Denote $\widetilde\theta_{\varphi}(F)$ the automorphic form on $D_\A^\times\times D_\A^\times$ obtained from
$\theta_{\varphi}(F)$ by pull-back along the map $\varrho$ in \eqref{eq:GODstr}.
For later use, note that a straightforward substitution yields the formula
\begin{equation}
     \label{eq:substform}
     \theta_\varphi(F)(hh_1)=\theta_{\tilde r_\psi(g_1,h_1)\varphi}(\pi(g_1)F),
\end{equation}
valid for all $(g_1,h_1)\in R(D)$.
Likewise, if $G$ is an automorphic form on $\GO(D)$ the function defined on the subgroup
$\GL_2(\A)^+$ consisting of the $g\in\GL_2(\A)$ such that $(g,h)\in R(D)_\A$ for some $h\in\GO(V)(\A)$  as
\begin{equation}
    \theta_\varphi^t(G)(g)=\int_{\text{O}(V)(\Q)\backslash\text{O}(V)(\A)}
    \vartheta(g,h^\prime h;\varphi)G(h^{\prime}h)\,dh^{\prime}
    \label{eq:thetaliftt}
\end{equation}
and  extended first to $\GL_2(\Q)\GL_2(\A)^+$ by left $\GL_2(\Q)$-invariance and set equal to $0$ outside 
is an automorphic form on $\GL_2(\A)$.

Let $f^\De\in M_{0,2\ka}^\De(N,\vep)$ be a newform with associated automorphic 
representation $\pi^D$ of $D_\A^\times$ with central character $\vep$. 
Denote $\pi=\textrm{JL}(\pi^D)$ the automorphic representation of $\GL_2(\A)$ corresponding to 
$\pi^D$ under the Jacquet-Langlands correspondence.
Let $\widetilde{\Theta}(\pi)=\{\widetilde\theta_{\varphi}(F)\}_{F\in\pi,\varphi\in\calS(D)_\A}$ 
be the space of automorphic forms on $D^{\times}\times D^{\times}$ which are 
pull-backs of theta lifts \eqref{eq:thetalift} from $\pi$. As pointed out in \cite{Harris93, HaKu91}
the work of Shimizu \cite{Shimi72} yields
\begin{equation}
\widetilde{\Theta}(\pi)=\pi^{D}\otimes\check{\pi}^{D},
\label{thm:Shimizu}
\end{equation}
where $\check{\pi}^{D}\simeq\pi^D\otimes(\inv\vep\circ\nu)$ is the contragredient 
representation of $\pi^{D}$. 
If $F$ and $\varphi$ are chosen so that
$\widetilde\theta_{\varphi}(F)=\pi(d^\tau_\infty)F^{(r)}_{f^\De}\otimes\pi(d^\tau_\infty)F^{(r)}_{f^\De}\cdot(\inv\vep\circ\nu)$ then
\begin{equation}
\label{eq:Lsquare}
L_{\underline\xi}(\widetilde\theta_{\varphi}(F)):=
\int_{(K^\times\R^\times\backslash K_\A^\times)^2}\widetilde\theta_{\varphi}(F)\underline\xi(t)\,dt
=J_r(f^\De,\xi,\tau)^2,
\end{equation}
where the pair $\underline\xi=(\xi,\xi(\vep\circ N))$ can be thought of as a character of
$K_\A^\times\times K_\A^\times$. 
On the other hand, Harris and Kudla observe that for all $F$ and $\varphi$ the integral 
$L_{\underline\xi}(\widetilde\theta_{\varphi}(F))$ can be read, via the map 
$\al$ in \eqref{eq:GOD} as the Petersson scalar product of two automorphic forms on the group
$T$, namely 
$L_{\underline\xi}(\widetilde\theta_{\varphi}(F))=
\int_{T(\Q)T(\R)\backslash T(\A)}\widetilde\theta_{\varphi}(F)(a,b)\xi(b)\,d^\times ad^\times b$,
where $\al(t)=(a,b)$.
Thus, the seesaw identity \cite[1.14]{Kudla84} associated with the seesaw dual pair
$$
\xymatrix{
\GL_2\times\GL_2\ar@{-}[dr] & \GO(D)\ar@{-}[dl] \\
\GL_2\ar[u] & T\ar[u]}
$$
identifies, up to a renormalization of the Haar measures, 
the value $L_{\underline\xi}(\widetilde{\theta}_{\varphi}(F))$ with a scalar 
product on $\GL_{2}$,
\begin{equation}
	L_{\underline{\xi}}(\widetilde{\theta}_{\varphi}(F))=
	\int_{\GL_{2}(\Q)\A^{\times}\backslash\GL_{2}(\A)}
	F(g)\theta_{\varphi}^t(1,\xi)(g,g)\,dg,
	\label{eq:RankinSelberg}
\end{equation}
where $\theta_{\varphi}^t(1,\xi)$ denotes the theta lift from $T$ to $\GL_{2}\times\GL_{2}$
which is defined in a fashion similar to \eqref{eq:thetalift}.
If $\varphi=\varphi_{1}\otimes\varphi_{2}$ is split  under a decomposition
$D(\A)=(K\oplus K^\perp)\otimes\A$ induced from an embedding $\rho:K\hookrightarrow D$
then $\theta_{\varphi}^t(1,\xi)$ is a product of two separate lifts:
$\theta_{\varphi}^t(1,\xi)(g,g)=E(0,\Phi,g)\theta_{\varphi_{2}}(\xi)(g)$ where 
\begin{itemize}
  \item $\theta_{\varphi_{2}}(\xi)(g)=\int_{K^{(1)}\backslash K_\A^{(1)}}\vartheta(g,kk^\prime;\varphi_2)
           \xi(kk^\prime)\,dk^\prime$ (for all $g\in\GL_2$ and $k\in K^\times_\A$ with $Nk=\det g$) 
           is the binary form in the automorphic representation 
           $\pi(\xi)$ of $\GL_{2}$ attached to $\xi$;
  \item $E(0,\Phi,g)$ is the value at $s=0$ of the holomorphic Eisenstein series attached to the unique 
           flat section (\cite[\S3.7]{Bump96}) extending the function $\Phi(g)=r_\psi(g,k)\varphi_1(0)$ 
           where $k\in\ideles$ is such that $N(k)=\det(g)$ and $r_\psi$ denotes here the extended adelic 
           Weil representation attached to $K$ as a normed space (Siegel-Weil formula).
\end{itemize}
This expression yields a relation between the right hand side 
of \eqref{eq:RankinSelberg} and the value at the centre of symmetry of a
Rankin-Selberg convolution integral. Furthermore, if each Schwartz-Bruhat function is primitive,
i.e. decomposes as a product
of local factors $\varphi_i=\bigotimes_{\ell\leq\infty}\varphi_{\ell,i}$ ($i=1$, $2$), and if
the Whittaker function $W_F$
of $F$ decomposes as a product of local Whittaker functions, 
the Rankin-Selberg integral admits an Euler 
decomposition \cite[\S19]{Ja72}. Precisely (see also \cite[\S 3.2]{Pra06})
$$
L_{\underline\xi}(\widetilde{\theta}_{\varphi}(F))=\prod_{\ell\leq\infty}
\left.L_\ell(\varphi_\ell,\xi_\ell,s)\right|_{s=1/2},
$$
where $L_\ell(\varphi_\ell,\xi_\ell,s)$ is the analytic continuation of
\begin{multline}
\int_{K_\ell}\int_{\Q_\ell^\times}
W^{\psi_\ell}_{F,\ell}\left(\left(\begin{array}{cc}a & 0 \\0 & 1\end{array}\right)k\right)
W^{\psi_\ell}_{\theta_{\varphi_{2},\ell}}\left(\left(\begin{array}{cc}-a & 0 \\0 & 1\end{array}\right)k\right)\cdot\\
\Phi^s_{\varphi_1,\ell}\left(\left(\begin{array}{cc}a & 0 \\0 & 1\end{array}\right)k\right)\inv{\vass{a}}\,
d^\times a\,dk_\ell
\label{eq:localterm}
\end{multline}
with the local measures normalized so that $K_\infty=\SO_2(\R)$ has volume $2\pi$ and
$K_\ell=\GL_2(\Z_\ell)$ has volume $1$ for finite $\ell$, 
$W_{\theta_{\varphi_{2}}}$ is the Whittaker function
and $\Phi^s(g)=\vvass{a}^{s-\frac12}\Phi(g)$ under the $NAK$-decomposition $g=nak$ 
with $\vvass{\smallmat a{}{}b}=\vass{a/b}$.
The local term \eqref{eq:localterm} does not vanish and for almost all $\ell$ can be recognized as a quotient of
Euler factors. Thus one obtains, 
as in \cite[2.3.11]{Harris93} and \cite[\S14]{HaKu91}, a version of Waldspurger's result \cite{Waldsp85}. 
Namely,
\begin{equation}
\label{eq:WHK}
   L_{\underline{\xi}}(\widetilde{\theta}_{\varphi}(F))=\left.
   \Lambda(\varphi,\xi,s)L(\pi_K\otimes\xi,s)L(\eta_K,2s)^{-1}\right|_{s=1/2},
\end{equation}
where $\Lambda(\varphi,\xi,s)$ is a finite product of local integrals, $\pi_K$ is the base change
to $K$ of the automorphic representation $\pi$ and $L(\eta_K,2s)$ is the Dirichlet $L$-function 
attached to $\eta_K$, the quadratic character associated to $K$. Note that the single local terms
\eqref{eq:localterm} do depend on the choice of $u\in K^\perp$.

\subsection{Making the theta correspondence explicit}\label{sec:thexpl}\label{sec:explicit}

Let us fix complex coordinates in $D_\infty$
%$D^\tau_\infty=(K\oplus Ku)\otimes\R$
as follows. The normalized standard embedding 
$\rho^{\rm st}:\Q(\sqrt{-1})\hookrightarrow{\rm M}_2(\Q)$ with 
fixed point $i\in\gotH$ defines a splitting $\Phi_\infty(D_\infty)=\C\oplus\C^\perp$ with
$\C=\R\smallmat1{}{}1\oplus\R\smallmat{}{-1}1{}$ and
$\C^\perp=\C\smallmat1{}{}{-1}$.  Thus, define standard complex coordinates $z_i=z_i^{\rm st}$, $i=1,2$
for $d\in D_\infty$  as 
$$
\Phi_\infty(d)=z^{\rm st}_1(d)+z^{\rm st}_2(d)\smallmat1{}{}{-1}.
$$
%The $\R$-linear extensions of the embeddings $\rho^{\rm st}$ and $\rho$ 
%are conjugated in $M_2(\R)$,
%precisely $\Phi_\infty\circ\rho=d^\tau_\infty(\Phi_\infty\circ\rho^{\rm st})(d^\tau_\infty)^{-1}$. Thus
%we finally define complex coordinates in $D^\tau_\infty$ as
%$z_i(d)=z_i^{\rm st}(\inv\Phi(d^\tau_\infty\inv)d\inv\Phi(d^\tau_\infty))$.
We choose $\varphi^\prime=\bigotimes_{v}\varphi^\prime_\ell$ 
with local components given as follows.
\begin{itemize}
  \item $\varphi^\prime_\infty(z_1,z_2)=\frac1\pi z_2^{2\ka}e^{-2\pi(z_1\bar z_1+z_2\bar z_2)}$;
  \item $\varphi^\prime_\ell=\frac1{\textrm{vol}(\calR_{N,\ell}^\times)}\textrm{ch}(\calR_{N,\ell})$
           for all finite $\ell$ such that $\vep_\ell$ is unramified;
  \item for all $\ell\mid N$ such that $\vep_\ell$ is ramified and of conductor $\ell^s$
           let $\varphi^\prime_\ell$ be the function supported on
           $S_\ell=\smallmat{\ell^s\Z_\ell}{\Z_\ell^\times}{\ell^s\Z_\ell}{\ell^s\Z_\ell}$ such that
           $$
           \varphi^\prime_\ell\smallmat{\ell^sa}b{\ell^sc}{\ell^sd}=
           \frac1{\textrm{vol}(\calR_{N,\ell}^\times)}\vep_\ell^{-1}(b),\qquad
           \forall\smallmat{\ell^sa}b{\ell^sc}{\ell^sd}\in S_\ell.
           $$
\end{itemize}
Note that for the archimedean prime and the finite primes for which $\vep$ is unramified, 
our choice of test function $\varphi^\prime_\ell$ coincides with or is an immediate generalization of
that of  Prasanna \cite{Pra06} (itself an adaptation of \cite{Wat03}) where $N$ is assumed square-free 
and $\vep$ unramified.

Let $F\in\pi$ be the lift of the holomorphic newform.

\begin{pro}
    $\widetilde{\th}_{\varphi^\prime}(F)=F^D\otimes\check{F}^D=F^D\otimes F^D(\vep^{-1}\circ\nu)$
    where $F^D$ is an holomorphic eigenform in $\pi^D$ such that 
    $\langle F,F\rangle=\langle F^D,F^D\rangle$.    
\end{pro}

\begin{proof}
Let $(d_1,d_2)\in D_\A^\times\times D_\A^\times$ and 
$(r_1,r_2)\in\calR_{N,\ell}^\times\times\calR_{N,\ell}^\times$.
It follows from \eqref{eq:substform} that
$\widetilde{\th}_{\varphi^\prime}(F)(d_1r_1,d_2r_2)=
\widetilde\th_{r_\psi(g,h)\varphi^\prime}(\pi(g)F)(d_1,d_2)$
for $h=\varrho(r_1,r_2)$ and any $g\in\GL_2(\A)$ such that $\det(g)=\nu_0(h)$. 
Clearly we can take $g$ to have components 
$1$ away from $\ell$ and, since $\nu_o(h)$ is a $\ell$-unit, $g_\ell=\smallmat1{}{}{\nu_o(h)}$. Thus,
$\pi(g)F=F$ and $r_\psi(g,h)\varphi^\prime(d)=\varphi^\prime(r_1^{-1}dr_2)$. 
It is now completely straightforward to check that
$$
\widetilde{\th}_{\varphi^\prime}(F)(d_1r_1,d_2r_2)=
\vep(\al_1)\vep^{-1}(\de_2)\widetilde{\th}_{\varphi^\prime}(F)(d_1,d_2), 
\quad
\text{if $r_i=\smallmat{\al_i}{\be_i}{\ga_i\ell^n}{\de_i}$, $i\in\{1,2\}$}.
$$
Given \eqref{thm:Shimizu}, Casselman's theory now applies to conclude that at every finite prime
$\widetilde{\th}_{\varphi^\prime}(F)$ is (proportional to) the tensor product of the newvectors.
Holomorphicity in both variables is obtained from the choice of $\varphi^\prime_\infty$ 
as in \cite{Pra06}, where the original argument in \cite{Wat03} is modified.
The identity of Petersson norms can also be obtained as in  \cite{Pra06}.
\end{proof}

\begin{rem}\label{rem:expth}\rm
  Applying \eqref{eq:substform} again finally yields 
  $$
  \widetilde{\th}_{\varphi^\sharp}(F)=\pi(d^\tau_\infty)F^D\otimes\pi(d^\tau_\infty){\check F}^D
  $$
  where $\varphi^\sharp=r_\psi(1,(d^\tau_\infty,d^\tau_\infty))\varphi^\prime$. Clearly $\varphi^\sharp$ differs from 
  $\varphi^\prime$ only for the archimede\-an component. Since $\rho=d^\tau_\infty\rho^{\rm st}(d^\tau_\infty)^{-1}$,
  if $D_\infty\ni x=\rho(w_1)+\rho(w_2)u$
  with $w_1,w_2\in K\otimes\R=\C$,
  $$
  \varphi^\sharp_\infty(x)=\frac1\pi(w_2\ze_u)^{2\ka}e^{-2\pi(w_1\bar w_1+\mid\nu(u)\mid w_2\bar w_2)}
  $$
  where $\ze_u\in\C$ is such that $(d_\infty^\tau)^{-1}ud_\infty^\tau=\ze_u\smallmat1{}{}{-1}$. In particular 
  $\varphi^\sharp_\infty$ is split under the decomposition of $D_\infty$ induced by $\rho$.
\end{rem}
%Finally, let $\varphi=\bigotimes_v\varphi_v\in\calS(D)$ where each local function $\varphi_v$ is the same as
%$\varphi^\sharp_v$ with the $\frac1\pi$ or the volumetric normalizing factor removed. Thus
%$$
%\varphi^\sharp=\frac1\pi C_{N,\De}\varphi,\quad
%C_{N,\De}=\prod_\ell{\rm vol}(\calR_{N,\ell}^\times)^{-1}=
%\ze(2)\prod_{\ell\mid N}\ell^{v_\ell(N)-1}(\ell+1)\prod_{\ell^\prime\mid\De}(\ell-1).
%$$

\bigskip
The following lemma allows us to factor the local components $\varphi_\ell$ in terms of the splitting 
$D_\ell=(K\oplus Ku)\otimes\Q_\ell$ induced by an embedding $\rho$.

\begin{lem}
  Let $L\subset D$ be an embedding of conductor $c$ with respect  to $\calR_N$
  of the quadratic field $L$  with associated decomposition $D=L\oplus Lu$. 
  Let $\La=L\cap\calR_N$,
  $\La^\prime=Lu\cap\calR_N$. Then
  \begin{enumerate}
  \item $D$ is split at $\ell$ if and only if $(u^2,\de_L)_\ell=1$.
  \item If $\ell$ is unramified in $L$ and $(\ell,c)=1$ then 
            $\calR_N\otimes\Z_\ell=\La\otimes\Z_\ell\oplus\La^\prime\otimes\Z_\ell$. 
            Moreover $\La^\prime\otimes\Z_\ell={\cal I}_\ell u$
            for some fractional ideal ${\cal I}_\ell\subset L\otimes\Q_\ell$ such that 
            ${\rm N}({\cal I_\ell})\nu(u)=\ell^{v_\ell(N)}$.
\end{enumerate}
\end{lem}

\begin{proof}
The first statement follows from the characterization in terms of the Hilbert symbol of the representability 
of $0$ by the norm form on $D\otimes\Q_\ell$.
To prove the second statement observe first that if $\ell$ is unramified in $L$ and coprime with $c$ then 
$e_\rho(\calR_N)\subset\calR_{N,\ell}$ (this follows readily from the explicit expression of $e_\rho$ 
of \S\ref{se:cmpts}). Then the last part follows from a local discriminant computation.
\end{proof}

\bigskip
The lemma yields immediately a decomposition 
$\varphi_\ell=\varphi_{\ell,1}\otimes\varphi_{\ell,2}$ for all finite $\ell$ such that
$(\ell,c\de_K N)=1$ and those $\ell\mid N$ for which $\vep_\ell$ is unramified, where 
$\varphi_{\ell,1}=\textrm{ch}(\calO_K\otimes\Z_\ell)$ and
 $\varphi_{\ell,2}=\textrm{ch}({\cal I}_\ell u)$.

When $\ell\mid\de_K$ the local component $\varphi_\ell$ does not decompose as a product. 
Write $K\otimes\Q_\ell=\Q_p(\varpi_\ell)$ with $\varpi_\ell^2=\lambda_\ell\ell$ 
with $\lambda_\ell$ a $\ell$-unit; also
$\calO_K\otimes\Z_\ell=\Z_\ell[\varpi_\ell]$.
Following \cite[\S3.4.5]{Pra06} consider the standard local embedding 
$\rho_{\rm st}:\Q_p(\varpi_\ell)\hookrightarrow D\otimes\Q_\ell\simeq{\rm M}_2(\Q_\ell)$
given by $\rho_{\rm st}(\varpi_\ell)=\smallmat01{\la_\ell\ell}0$. 
The embeddings $\rho$ and  $\rho_{\rm st}$
are conjugated by a matrix which can be proved to be in ${\rm GL}_2(\Z_\ell)$.
Thus, if $\la\in K$ is such that $u$ is conjugated to $\rho_{\rm st}(\la)\smallmat100{-1}$
one eventually gets
$$
\varphi_\ell=
\begin{cases}
\sum_{j=0}^{\ell-1}\varphi_{\ell,1,j}\otimes\varphi_{\ell,2,j}     & \text{if $\ell>2$}, \\
\sum_{i=0}^{1}\sum_{j=0}^3\varphi_{2,1,(i,j)}\otimes\varphi_{2,2,(i,j)}    & \text{if $\ell=2$}
\end{cases}
$$
with $\varphi_{\ell,1,j}=\textrm{ch}(\Z_\ell[\varpi_\ell]+\frac j{\varpi_\ell})$, 
$\varphi_{\ell,2,j}=\textrm{ch}\big((\Z_\ell[\varpi_\ell]+\frac j{\varpi_\ell})(\inv\la{\cal I})\big)$
if $\ell>2$, and
$$
\left\{
\begin{array}{l}
\varphi_{2,1,(i,j)}=\textrm{ch}\left(\Z_2[\varpi_2]+\frac12\big(i+ \frac j{\varpi_2}\big)\right)\\
\varphi_{2,2,(i,j)}=\textrm{ch}\left(\big(\Z_2[\varpi_2]+\frac 12(i+\frac j{\varpi_2})\big)(\inv\la{\cal I})\right).
\end{array}
\right.
$$
\bigskip
We now deal with $\varphi_\ell$ for primes $\ell\mid N$. Since these are split in $K=\Q(\sqrt{d})$
we can fix a $\de_\ell\in\Z_\ell^\times$ such that $\de_\ell^2=d$ and if 
$\sqrt{d}\mapsto\smallmat ab{c\ell^{v_\ell(N)}}{-a}\in\calR(\ell^{v_\ell(N)})\otimes\Z_\ell$ ($bc\neq0$)
after localizing $\rho$ at $\ell$ and we  may also assume that $\frac12(\de_\ell+a)\in\Z_\ell^{\times}$.

\begin{lem}
   Let $\ell\mid N$. With the notation just introduced, there exists $u_\ell\in\rho(K)^\perp\otimes\Q_\ell$
   such that $\varphi_\ell=\varphi_{\ell,1}\otimes\varphi_{\ell,2}$ where
   $\varphi_{\ell,1}={\rm ch}(\ell^{v_\ell(N)}\calO_K\otimes\Z_\ell)$ and
   $\varphi_{\ell,2}$ is the function supported on 
   $(\ell^{-v_\ell(b)}\Z_\ell^\times\times\ell^{-v_\ell(c)}\Z_\ell)u_\ell$ such that
   $\varphi_{\ell,2}((x,y)u_\ell)=\inv\vep_\ell(-x(a+\de_\ell)b/2)$.
\end{lem}

\begin{proof}
The idempotents $e^{\pm}=\frac12(1\otimes1\pm\sqrt{d}\otimes\inv\de_\ell)$ 
define an isomorphism of rings 
$\calO_K\otimes\Z_\ell=\Z_\ell e^+\times\Z_\ell e^-\simeq\Z_\ell\times\Z_\ell$
(this is true also if $\ell=2$ as 
$e^{\pm}=1\otimes\frac{\de_2\pm1}{2\de_2}\pm\frac{1+\sqrt{d}}2\otimes\de_2
\in\calO_K\otimes\Z_2$ since $\de_2\equiv1\bmod2$).
Obviously $\rho(e^{\pm})=M^\pm=
\frac1{2\de_\ell}\smallmat{\de_\ell\pm a}{\pm b}{\pm cp^{v_\ell{N}}}{\de_\ell\mp a}$.
To get an explicit description of $\rho(K)^\perp\otimes\Q_\ell$
we can choose any $u_\ell\in\rho(K)^\perp\otimes\Q_\ell$ which generates it as a 
$\Q_\ell\times\Q_\ell$-module. Choose
$$
u_\ell=(1-e_\rho)\left(\begin{array}{cc}1 & 0 \\0 & -1\end{array}\right)=
\frac1d\left(\begin{array}{cc}d-a^2 & -ab \\-ac\ell^{v_\ell(N)} & a^2-d\end{array}\right).
$$
Under our assumptions
$\inv{(\ell^{v_\ell(b)})}M^+u_\ell=\frac1{2d\ell^{v_\ell(b)}}
\smallmat{d-a^2}{-b(a+\de_\ell)}{c(\de_\ell-a)\ell^{v_\ell(N)}}{a^2-d}$ and
$\inv{(\ell^{v_\ell(c)})}M^-u_\ell=\frac1{2d\ell^{v_\ell(c)}}
\smallmat{d-a^2}{b(\de_\ell-a)}{c(\de_\ell+a)\ell^{v_\ell(N)}}{a^2-d}$
are a $\Z_\ell$-basis of $(\rho(K)^\perp\otimes\Q_\ell)\cap\calR_{\ell^{v_\ell(N)},\ell}$.
It is now straightforward to check that given $(x^+,x^-,y^+,y^-)\in\Z_\ell^4$,
$x^+M^++x^-M^-+y^+\ell^{-v_\ell(b)}M^+u_\ell+y^-\ell^{-v_\ell(c)}M^-u_\ell\in S_\ell$
if and only if $x^\pm\in\ell^{v_\ell(N)}\Z_\ell$, $y^+\in\Z_\ell^\times$ and 
$y^-\in\Z_\ell$. The final formula follows.
\end{proof}

\subsection{Proof of main theorem \ref{th:main}}

Let $D$ be the quaternion algebra with discriminant $\De$. By hypothesis $\pi$ is  the image of an automorphic
representation $\pi^D$ of $D^\times$ under the Jacquet-Langlands correspondence. Let $f^\De\in S^\De_{2k}(N_0,\vep)$
be a holomorphic newform in $\pi^D$ which we may assume to be defined over $\calO_{(v)}$. 
For all $r\geq0$ let $\phi^{(r)}_{f^\De}$ be as in \eqref{def:phir}
Following remark \ref{rem:expth} write 
$\pi(d^\tau_\infty)\phi_{f^\De}\otimes\pi(d^\tau_\infty)\phi_{f^\De}\cdot(\inv\vep\circ\nu)=\ups_{f^\De}\tilde\th_{\varphi^\sharp}(F)$
where $F$ is the adelic lift of a normalized eigenform in $\pi$, $\ups=\ups_{f^\De}\in\C^\times$ a normalizing factor 
(a quotient of Petersson norms) and 
$\varphi^\sharp=\varphi^\sharp_\infty\otimes\varphi^\sharp_{\rm fin}=\bigotimes_v\varphi^\sharp_v$.

Under the Lie algebra identifications $\gotg\gotoo(D)\simeq(D_\infty\times D_\infty)/\R$ and
$\gotoo(D)\simeq\gots\gotl_2\times\gots\gotl_2$ induced by 
the exact sequence \eqref{eq:GODstr}, the substitution formula \eqref{eq:substform} yields
$$
\tilde\theta_{H\varphi_\infty\otimes\varphi_{\rm fin}}(F)(h)=\left.\frac d{dt}\right|_{t=0}
\tilde\th_{\varphi_\infty\otimes\varphi_{\rm fin}}(F)(h\exp(tH))
$$ 
where
$H\varphi_\infty(x)=\left.\frac d{dt}\right|_{t=0}\varphi_\infty(e^{-tH_1}xe^{tH_2})$
for all $H=(H_1,H_2)\in\gotoo(D)$. In turn, this implies that the diagonal action of $A\in\gots\gotl_2$ on 
$\pi^D\otimes\check\pi^D$ corresponds to the second order operator $A^\prime=(A,0)\circ(0,A)\in\gotA(\gotoo(D))$ 
on Schwartz functions, i.e. 
$A^\prime\varphi_\infty(x)=\left.\frac{\partial^2}{\partial u\partial v}\right|_{u=v=0}\varphi_\infty(e^{-uA}xe^{vA})$.
Note that since ${\rm tr}(A)=0$ this action leaves unchanged the factor $\inv\vep\circ\nu$.
Therefore
$$
\pi(d^\tau_\infty)\phi^{(r)}_{f^\De}\otimes\pi(d^\tau_\infty)\phi^{(r)}_{f^\De}\cdot(\inv\vep\circ\nu)=\
\ups\tilde\th_{(X^\prime_\tau)^r\varphi_\infty^\sharp\otimes\varphi^\sharp_{\rm fin}}(F).
$$
where $X^\prime_\tau=(-\frac1{4\pi}d_\infty^\tau X^+(d_\infty^\tau)^{-1})^\prime$. Since
\begin{eqnarray*}
X_\tau^\prime\varphi^\sharp_\infty(x) & = & \left.\frac{\partial^2}{\partial u\partial v}\right|_{u=v=0}\varphi^\sharp_\infty(e^{-ud_\infty^\tau X^+(d_\infty^\tau)^{-1}}x
                                                                      e^{vd_\infty^\tau X^+(d_\infty^\tau)^{-1}}) \\
   & = & \left.\frac{\partial^2}{\partial u\partial v}\right|_{u=v=0}\varphi^\sharp_\infty(d_\infty^\tau e^{-uX^+}(d_\infty^\tau)^{-1}xd_\infty^\tau e^{vX^+}(d_\infty^\tau)^{-1}) \\
   & = & (X^+)^\prime\varphi^\prime_\infty((d_\infty^\tau)^{-1}xd_\infty^\tau)
\end{eqnarray*}
the computation of the operator $X^\prime_\tau$ reduces to that of $(X^+)^\prime$.
Let $\mathfrak h=\R\smallmat{}{-1}1{}$ be the standard Cartan subalgebra of $\mathfrak s\mathfrak l_2$. 
The product $\mathfrak h\times\mathfrak h$
is identified to a  Cartan subalgebra of $\mathfrak o(D)$ and if $e_j\in i\mathfrak h^\ast$ is the standard half-root of 
the $j$-th factor then we have the following table of roots and eigenvectors for the action of $\mathfrak o(D)$ on $D_\infty$:

\begin{center}
\begin{tabular}{|c|c|c|c|}\hline
% after \\ : \hline or \cline{col1-col2} \cline{col3-col4} ...
  $e_1+e_2$ & $e_1-e_2$ & $-e_1+e_2$ & $-e_1-e_2$  \\
  \hline
  $E_{++}=\smallmat{-i}{1}{1}{i}$  &   $E_{+-}=\smallmat{-i}{-1}{1}{-i}$ &  $E_{-+}=\smallmat{i}{-1}{1}{i}$ &  $E_{--}=\smallmat{i}{1}{1}{-i}$\\
  \hline
\end{tabular}
\end{center}

\smallskip\noindent
The action of the operators $(X^+,0)$ and $(0,X^+)$ on these eigenvectors is
$$
\begin{array}{cc}
   (X^+,0)E_{+\pm}=0  & (X^+,0)E_{-\pm}=-E_{+\pm}    \\
   (0,X^+)E_{\pm,+}=0  & (0,X^+)E_{\pm-}=-E_{\pm +}    
\end{array}
$$
Comparing the expressions of $x\in D_\infty$ in terms of the complex standard coordinates $z_i=z_i^{\rm st}$
and in the real coordinates with respect to the above eigenvectors allows to compute the expression of $(X^+,0)$ and $(0,X^+)$ 
in terms of the standard coordinates. Eventually
$$
(X^+)^\prime=
z_2^2\frac{\partial^2}{\partial z_1\partial{\bar z}_1}+
{\bar z}_1{z}_2\frac{\partial^2}{\partial{\bar z}_1\partial{\bar z}_2}+
z_1z_2\frac{\partial^2}{\partial z_1\partial{\bar z}_2}+
z_1{\bar z}_1\frac{\partial^2}{{\partial{\bar z}_2}^2}+
z_2\frac{\partial}{\partial{\bar z}_2}.
$$
For a pair $(l,m)$ of non-negative integers, consider the function of two complex variables
$\varphi^{(l,m)}(z_1,z_2)=(z_1\bar z_1)^lz_2^{2m}e^{-2\pi(z_1\bar z_1+z_2\bar z_2)}$. Then
$$
(X^+)^\prime\varphi^{(l,m)}=l^2\varphi^{(l-1,m+1)}-4\pi(2l+1)\varphi^{(l,m+1)}+(4\pi)^2\varphi^{(l+1,m+1)}.
$$
and a $r$-fold iteration allows to conclude eventually that
\begin{multline}\label{eq:calctheta}
\pi(d^\tau_\infty)\phi_f^{(r)}\otimes\pi(d^\tau_\infty)\phi_f^{(r)}(\inv\vep\circ\nu)=\\
\ups\left(\tilde\theta_{\varphi^{\sharp (r,r)}\otimes\varphi^\sharp_{\rm fin}}(F)+
\sum_{l=0}^{r-1}a_{r,l}\tilde\theta_{\varphi^{\sharp (l,r)}\otimes\varphi^\sharp_{\rm fin}}(F)\right)
\end{multline}
%$\pi(d^\tau_\infty)\phi_f^{(r)}\otimes\pi(d^\tau_\infty)\phi_f^{(r)}(\inv\vep\circ\nu)$ is
for some $a_{r,l}\in\Z[1/4\pi]$ and $\varphi^{\sharp (l,r)}(x)=\frac1\pi\varphi^{(l,\ka+r)}((d_\infty^\tau)^{-1}xd_\infty^\tau)$.
The function $\varphi^{\sharp,l}$ is split under the decomposition of $D_\infty\ni x=\rho(z_1)+\rho(z_2)u$
induced by $\rho$, namely $\varphi^{\sharp (l,r)}=\frac1\pi\varphi^{\sharp,l}_1\otimes\varphi^{\sharp,r}_2$ with
$$
\varphi^{\sharp,l}_1(x)=(z_1\bar z_1)^le^{-2\pi iz_1\bar z_1},\qquad
\varphi^{\sharp,r}_2(x)=\ze_u^{2(\ka+r)}z_2^{2(\ka+r)}e^{-2\pi i|\nu(u)|z_2\bar z_2}.
$$

\begin{lem}
\begin{multline}
L_\infty(\varphi^{\sharp(l,r)},\xi_{r,\infty},s)=\\
=\begin{cases}
    2\frac{\overline{j(u,\tau)}^{\ka+r-\frac12}}{j(u,\tau)^{\ka+r+\frac12}}\frac{r!}{(4\pi)^{s+2(\ka+r)+\frac12}}\Ga(s+2\ka+r+\frac12) &  
    \text{if $l=r$}, \\
    0  & \text{if $0\leq l<r$}.
\end{cases}
\end{multline}
\end{lem}

\begin{proof}
We shall compute separately the three main ingredients of the integrand in the local $L_\infty$ factor and piece
them together later.

It is well known that
$W_F^{\psi_\infty}\left(\smallmat a{}{}1r_\th\right)={\rm ch}_{\R^\times}(a)a^\ka e^{-2\pi a-2\ka i\th}$.

To compute $\Phi^s_{\varphi^{\sharp,l}_1}\left(\smallmat a{}{}1r_\th\right)=|a|^sr_{\psi_\infty}(r_\th)\varphi^{\sharp,l}_1(0)$
we use the decomposition
\begin{multline}\label{eq:decrtheta}
r_\th=\left(\begin{array}{cc}1 & -\tan\th \\0 & 1\end{array}\right)\left(\begin{array}{cc}0 & 1 \\-1 & 0\end{array}\right)
\left(\begin{array}{cc}1 & -\sin\th\cos\th \\0 & 1\end{array}\right)\cdot\\
\left(\begin{array}{cc}0 & 1 \\-1 & 0\end{array}\right)
\left(\begin{array}{cc}1/\cos\th & 0 \\0 & \cos\th\end{array}\right).
\end{multline}
Applying repeatedly %ADD RED TO WEIL REPRESENTATION
we get $r_{\psi_\infty}(r_\th)\varphi^{\sharp,l}_1(0)=-g(0)\cos\th$ where $g(z)$ is the Fourier transform of
${\hat\varphi}^{\sharp,l}_1(z\cos\th)e^{-2\pi i\cos\th\sin\th|z|^2}$. The Fourier transform ${\hat\varphi}^{\sharp,l}_1$ is given by
$$
{\hat\varphi}^{\sharp,l}_1(x_1+x_2i)=e^{-2\pi(x_1^2+x_2^2)}\sum_{0\leq\al+\be\leq l}\ga^l_{\al,\be}x_1^{2\al}x_2^{2\be}
$$
with coefficients 
$\ga^l_{\al,\be}=\sum_{}(-4\pi)^{\al+\be-l}\binom lj\binom{2j}{2\al}\binom{2k}{2\be}(2j-2\al-1)!!(2k-2\be-1)!!$.
Therefore
\begin{multline}
g(0)=\int_{\C}{\hat\varphi}^{\sharp,l}_1(z\cos\th)e^{-2\pi i\cos\th\sin\th|z|^2}\,dz \\
      = 2\sum_{0\leq\al+\be\leq l}\ga^l_{\al,\be}(\cos\th)^{2(\al+\be)}\int_{\R^2}e^{-2\pi(i\sin\th\cos\th+(\cos\th)^2)(x^2+y^2)}
              x^{2\al}y^{2\be}\,dxdy \\
      =-\sum_{0\leq\al+\be\leq l}\ga^l_{\al,\be}\frac{(2\al-1)!!(2\be-1)!!}{(4\pi)^{\al+\be}}
                   \frac{(\cos\th)^{2(\al+\be)}}{(-\sin\th\cos\th-(\cos \th)^2)^{\al+\be+1}}
\end{multline}
and eventually
\begin{multline*}
\Phi^s_{\varphi^{\sharp,l}_1}\left(\begin{pmatrix}a & 0 \\ 0 & 1 \end{pmatrix}r_\th\right)=\\
|a|^s\sum_{0\leq\al+\be\leq l}\frac{\ga^l_{\al,\be}(2\al-1)!!(2\be-1)!!}{(4\pi)^{\al+\be}}(\cos\th)^{\al+\be}e^{-(\al+\be+1)i\th}.
\end{multline*}

Also the computation of $W_{\theta_{\varphi_{2,\infty}^{\sharp,r}}}^{\psi_\infty}$ uses \eqref{eq:decrtheta} and goes as in
\cite{Pra06}: we summarize it briefly for the sake of completeness. Note first that the norm in $(Ku)\otimes\R\simeq\C$
is $-{\rm N}_{\C/\R}$ (definite negative) and the main involution is $z\mapsto-z$.
Thus $W_{\theta_{\varphi_{2,\infty}^{\sharp,r}}}^{\psi_\infty}\left(\smallmat{-a}001r_\th\right)=e^{i(2\ka+2r+1)}
W_{\theta_{\varphi_{2,\infty}^{\sharp,r}}}^{\psi_\infty}\left(\smallmat{-a}001\right)$.
We need $h\in\C^\times$ such that $|h|=-a\nu(u)^{-1}$: such $h$ exists only if $a>0$ and so we may pick
$h=\sqrt{a|\nu(u)|^{-1}}\in\R^\times$. We have
\begin{eqnarray*}
r_\psi\left(\begin{pmatrix} a|\nu(u)|^{-1} & 0 \\ 0 & 1\end{pmatrix},h\vth\right)\varphi^{\sharp,l}_2(u) & = & 
               \left(a|\nu(u)|^{-1}\right)^{\frac12}\varphi^{\sharp,l}_2(a|\nu(u)|^{-1}(h\vth)^{-1}u) \\
  & = &  \left(a|\nu(u)|^{-1}\right)^{\frac12}(\ze_uh\vth^{-1})^{2(\ka+l)}e^{-2\pi a}
\end{eqnarray*}
so that
\begin{eqnarray*}
W_{\theta_{\varphi_{2,\infty}^{\sharp,r}}}^{\psi_\infty}\begin{pmatrix}-a & 0\\ 0  & 1\end{pmatrix} & = & \int_{S^1}r_\psi
                \left(\begin{pmatrix} a|\nu(u)|^{-1} & 0 \\ 0 & 1\end{pmatrix},h\vth\right)\varphi^{\sharp,r}_2(u)\xi_{r.\infty}(h\vth)\,d\vth \\
  & = &  \ze_u^{2(\ka+l)}{\rm ch}_{\R^{>0}}(a)|\nu(u)|^{-\frac{2(\ka+r)+1}2}a^{\frac{2(\ka+r)+1}2}e^{-2\pi a}.
\end{eqnarray*}
Summing up:
\begin{multline*}
L_\infty(\varphi^{\sharp(l,r)},\xi_{r,\infty},s)=\frac1\pi\ze_u^{2(\ka+r)}{|\nu(u)|^{\frac{-2(\ka+r)-1}2}}\\
\times\int_{S^1}\int_0^\infty a^{2\ka+r+s-\frac12}e^{-4\pi a}
\sum_{0\leq\al+\be\leq l}\de^l_{\al,\be}(\cos\vth)^{\al+\be}e^{(2r-\al-\be)i\vth}\,dad\vth=\\
\frac{\xi_{\infty,r}(\ze_u)}{|\nu(u)|^{\frac12}}\frac{\Ga(2\ka+r+s+\frac12)}{\pi(4\pi)^{2\ka+r+s+\frac12}}
\sum_{0\leq\al+\be\leq l}\frac{\de^l_{\al,\be}}{2^{\al+\be}}\sum_{j=0}^{\al+\be}\binom{\al+\be}j
\int_{S^1}e^{2(r+j-\al-\be)i\vth}\,d\vth,
\end{multline*}
where $\de^l_{\al,\be}=\ga^l_{\al,\be}(2\al-1)!!(2\be-1)!!(4\pi)^{-\al-\be}$. Since the last integral is $0$ unless
$\al+\be=l=r$ and $j=0$, this proves the claim for $0\leq l<r$.

When $\al+\be=l=r$ a straightforward induction shows that 
$\sum_{\al=0}^r\de^r_{\al,r-\al}=(4\pi)^{-r}\sum_{\al=0}^r\binom r\al(2\al-1)!!(2r-2\al-1)!!=(4\pi)^{-r}2^rr!$ so that
$$
L_\infty(\varphi^{\sharp(r,r)},\xi_{r,\infty},s)=
2\frac{\xi_{\infty,r}(\ze_u)}{|\nu(u)|^{\frac12}}\frac{r!}{(4\pi)^{2(\ka+r)+s+\frac12}}\Ga(2\ka+r+s+\frac12).
$$
Applying the cochain relation for the automorphy factor $j(\ga,\tau)$
to both sides of the identity $(d_\infty^\tau)^{-1}ud_\infty^\tau=\ze_u\smallmat1{}{}{-1}$ and observing that $u\cdot\tau=\bar\tau$
one eventually obtains $\ze_u=-\overline{j(u,\tau)}$ and the formula follows.
\end{proof}

Since the archimedean component $(\pi_K)_\infty$ is the representation $\pi(\la,\la)$ of $\GL_2(\C)$
associated to the character $\la(z)=(z\bar z)^{-\ka+\frac12}z^{2\ka-1}$ we can write
$$
\left.L_\infty(\varphi^{\sharp(r,r)},\xi_{r,\infty},s)\right|_{s=\frac12}=
\left.\La_\infty(\varphi^{\sharp(r,r)},\xi_{r,\infty},s)L_\infty(\pi_K\otimes\xi_r,s)\right|_{s=\frac12}
$$
with 
$\La_\infty(\varphi^{\sharp(r,r)},\xi_{r,\infty},\frac12)=\frac12\frac{\overline{j(u,\tau)}^{\ka+r-\frac12}}{j(u,\tau)^{\ka+r+\frac12}}
r!(2\ka+r)^2\pi^{2\ka+r-1}\in\pi^{2\ka+r-1}\overline\Q$.
Combining the lemma with \eqref{eq:Lsquare} and \eqref{eq:WHK} yields
\begin{multline}\label{almostthere}
J_r(f^\De,\xi_r,\tau)^2=L_{\xi_r}(\tilde\th_{\varphi^{\sharp(r,r)}\otimes\varphi_{\rm fin}}(F))=\\
\ups\pi^{2\ka+r-1}\La(\varphi^{\sharp(r,r)}\otimes\varphi_{\rm fin},\xi_r,\frac12) 
L(\pi_K\otimes\xi_r,\frac12)L(\eta_K,1)^{-1}
\end{multline}
where the archimedean factor $\La_\infty$ has been redefined so to make explicit the correct power of $\pi$.

On the other hand, assuming $\calO_{K,c}^\times=\{\pm1\}$, from proposition \ref{pr:Jasscpr}
$$
J_r(f^\De,\xi_r,\tau)^2=\pi^2{\mathrm{vol}}(\hat\calO_{K,c}^\times)^2(h_c^\sharp)^2
(y\Om_\infty(\tau)^2)^{2(\ka+r)}\langle\vth(r,f^\De,x),\chi_r\rangle_\infty^2.
$$
When $\chi_r$ is the $r$-th moment of a $\gotF(c_\vep^*,F)$-valued measure on $\Z_p$
propositions \ref{pairedmeas} and \ref{pr:meastheta} apply to define a measure $\mu$ on $\Z_p$
such that
$$
\frac1{\Om_p^{4(\ka+r)}}\left(\int_{\Z_p}t^r\,d\mu(t)\right)^2=
\left(\frac1{\pi{\mathrm{vol}}(\hat\calO_{K,c}^\times)h_c^\sharp
(y\Om_\infty(\tau)^2)^{\ka+r}}\right)^2J_r(f^\De,\xi_r,\tau)^2.
$$
The last part of proposition 
\ref{pairedmeas} and corollary \ref{measthstar} yield
\begin{multline*}
\frac1{\Om_p^{4(\ka+r)}}\left(\int_{\Z_p^\times}t^r\,d\mu(t)\right)^2=\\
\left(\frac1{\pi{\mathrm{vol}}(\hat\calO_{K,c}^\times)h_c^\sharp
(y\Om_\infty(\tau)^2)^{\ka+r}}\right)^2J_r((1-VU)f^\De,\xi_r,\tau)^2.
\end{multline*}
Then one concludes applying corollary \ref{eulerfact}, plugging \eqref{almostthere} in and taking 
into account remark \ref{rem:twists}.

%\subsection{Some local integrals}
%
%The local Whittaker functions $W^{\psi_\ell}_{\th_{\phi_{2,\ell}}}$ have been computed in
%\cite[\S3.3]{Pra06}. The result is as follows.
%
%\begin{lem}
%    Suppose that $\tilde\xi$ is not trivial on $K^{(1)}_\A$. Then for all $\varphi_2\in\calS(K_\A^\perp)${}
%    $$
%    \theta_{\varphi_2}(\tilde\xi)(g)=\sum_{q\in\nu(u)N(K^\times)}
%    W^\psi_{\th_{\varphi_2}}\left(\left(\begin{array}{cc}q & 0 \\0 & 1\end{array}\right)g\right)
%    $$
%    where $W^\psi_{\theta_{\varphi_2}}(g)=\int_{K^{(1)}_\A}\tilde r_\psi\left(\smallmat{\inv{\nu(u)}}{}{}1g,kk^\prime\right)
%    \varphi_2(u)\xi(kk^\prime)\,dk$ for any choice of $k^\prime$ such that $Nk=\inv{\nu(u)}\det g$. If 
%    $\varphi_2=\otimes_\ell\varphi_{\ell,2}$ and $g=(g_\ell)$ then 
%    $W^\psi_{\varphi_2}(g)=\prod_\ell W^{\psi_\ell}_{\varphi_{2,\ell}}(g_\ell)$ where
%    $$
%    W^{\psi_\ell}_{\varphi_{2,\ell}}(g_\ell)=\int_{(K\otimes\Q_\ell)^{(1)}}\tilde r_{\psi_\ell}
%    \left(\left(\begin{array}{cc}\inv{\nu(u)} & 0 \\0 & 1\end{array}\right)g,kk^\prime\right)\varphi_2(u)\tilde\xi_\ell(kk^\prime)\,dk
%    $$
%    if $\det(g_\ell)=Nk^\prime$ for some $k^\prime\in K\otimes\Q_\ell$ or $W^{\psi_\ell}_{\varphi_{2,\ell}}(g_\ell)=0$
%    otherwise.
%\end{lem}

%%% BIBLIOGRAPHY

\end{document}